\documentclass{amsart}
\usepackage{amssymb, amsmath, amsthm, graphics, comment, xspace, enumerate}
\usepackage{color}
\newtheorem{thm}{Theorem}[section]
\newtheorem{cor}[thm]{Corollary}
\newtheorem{lem}[thm]{Lemma}
\newtheorem{prop}[thm]{Proposition}
\theoremstyle{definition}
\newtheorem{defn}[thm]{Definition}
\newtheorem{example}[thm]{Example}
\theoremstyle{remark}
\newtheorem{rem}[thm]{Remark}
\numberwithin{equation}{section}

\def\R{\mathbb R}

\def\N{\mathbb N}
\def\ep{\varepsilon}

\def\ga{\gamma }

\begin{document}
\title[Multi-dimensional $\rho$-almost periodic type functions...]{Multi-dimensional $\rho$-almost periodic type functions and applications}

\author[M. Fe\v{c}kan, M. T. Khalladi, M. Kosti\'{c} and A. Rahmani]{M. Fe%
\v{c}kan, M. T. Khalladi, M. Kosti\'{c} and A. Rahmani}
\address{Department of Mathematical Analysis and Numerical Mathematics,
Faculty of Mathematics, Physics and Informatics, Comenius University,
Slovakia}
\email{Michal.Feckan@fmph.uniba.sk}
\address{Department of Mathematics and Computer Sciences, University of
Adrar, Algeria}
\email{khalladimohammedtaha@yahoo.fr}
\address{Faculty of Technical Sciences, University of Novi Sad, Trg D.
Obradovica 6, 21125 Novi Sad, Serbia}
\email{marco.s@verat.net}
\address{Laboratory of Mathematics, Modeling and Applications (LaMMA),
University of Adrar, Adrar, Algeria}
\email{vuralrahmani@gmail.com}

{\renewcommand{\thefootnote}{} \footnote{2010 {\it Mathematics
Subject Classification.} 42A75, 43A60, 47D99.
\\ \text{  }  \ \    {\it Key words and phrases.} Multi-dimensional $\rho$-almost periodic functions, multi-dimensional $(\omega,\rho)$-almost periodic functions, 
multi-dimensional
$({\bf \omega}_{j},\rho_{j})_{j\in {\mathbb N}_{n}}$-periodic functions,
abstract Volterra integro-differential equations.
\\  \text{  }  
The third named author is partially supported by grant 451-03-68/2020/14/200156 of Ministry
of Science and Technological Development, Republic of Serbia.}}

\begin{abstract}
In this paper, we analyze various classes of multi-dimensional $\rho$-almost periodic type functions $F : I \times X \rightarrow Y$ and multi-dimensional $(\omega,\rho)$-almost periodic type functions $F : I \times X \rightarrow Y,$ where 
$n\in {\mathbb N},$
$\emptyset \neq I \subseteq {\mathbb R}^{n},$ $X$ and $Y$ are complex Banach spaces and $\rho$ is a binary relation on $Y.$ 
The proposed notion is new even in the one-dimensional setting, for the functions of the form $F : I \rightarrow Y.$
The main structural properties and characterizations for the introduced classes of functions are presented.
We provide certain applications of our theoretical results to
the abstract Volterra integro-differential equations, as well.
\end{abstract}
\maketitle

\section{Introduction and preliminaries}

The notion of an almost periodic function was introduced by H. Bohr \cite{h.bor} around 1924-1926 and later generalized by many other mathematicians. Let $I$ be either ${\mathbb R}$ or $[0,\infty),$ let $c\in {\mathbb C} \setminus \{0\}$ satisfy $|c|=1,$ and let $f : I \rightarrow X$ be a given continuous function, where $X$ is a complex Banach space equipped with the norm $\| \cdot \|$. For any $\varepsilon>0,$ a number $\tau>0$ is called a $(\varepsilon,c)$-period for $f(\cdot)$ if and only if\index{$\varepsilon$-period}
$
\| f(t+\tau)-cf(t) \| \leq \varepsilon,$ $ t\in I.
$
The set consisting of all $(\varepsilon,c)$-periods for $f(\cdot)$ is denoted by $\vartheta_{c}(f,\varepsilon).$ The function $f(\cdot)$ is said to be $c$-almost periodic if and only if for each $\varepsilon>0$ the set $\vartheta_{c}(f,\varepsilon)$ is relatively dense in $[0,\infty),$ i.e., 
there exists $l>0$ such that any subinterval of $[0,\infty)$ of length $l$ meets $\vartheta_{c}(f,\varepsilon);$ see \cite[Definition 2.1]{c1}, and the research articles \cite{cheban-anp,cheban-anp1} by D. Cheban and his coauthors for some pioneering results in this direction. The usual notion of almost periodicity (almost anti-periodicity) is obtained by plugging $c=1$ ($c=-1$); for more details about the subject, the reader may consult the research monographs \cite{besik}, \cite{diagana}-\cite{gaston}, \cite{nova-mono}-\cite{nova-mse}, \cite{188}-\cite{18}, \cite{pankov} and \cite{30}. 

The notion of a periodic function has recently been reconsidered by E. Alvarez, A. G\'omez and M. Pinto \cite{alvarez1} as follows: 
A continuous function $f : I \rightarrow X$ is said to be 
$(\omega,c)$-periodic ($\omega>0,$ $c\in {\mathbb C} \setminus \{0\}$) if and only if 
$f(t+\omega)=cf(t)$ for all $t\in I.$ 
It is well known that a continuous function $f : I \rightarrow X$ is $(\omega,c)$-periodic if and only if the function $g(\cdot)\equiv c^{-\cdot/\omega}f(\cdot)$ is periodic and 
$g(t+\omega)=g(t)$ for all $t\in I;$ here, $c^{-\cdot/\omega}$ denotes the principal branch of the exponential function (cf. also \cite{aga-feckan} and \cite{alvarez2}-\cite{alvarez3}). In \cite{c1-fil}, the authors have recently analyzed various classes of 
$(\omega,c)$-almost periodic type functions.

The concept of $( w,\mathbb{T}) $-periodicity for a continuous function $f : [0,\infty) \rightarrow X$,  where $\omega>0$ and $\mathbb{T} : X \rightarrow X$ is a linear
isomorphism, has recently been introduced and analyzed by M. Fe\v{c}kan, K. Liu and
J. Wang in \cite[Definition 2.2]{Reference040}; more precisely, a continuous function $f : [0,\infty) \rightarrow X$ is called $( w,\mathbb{T}) $-periodic if and only if 
$f(t+\omega)={\mathbb T}f(t)$ for all $t\geq 0.$ In the above-mentioned paper, the existence and uniqueness of $( w,\mathbb{T}) $-periodic solutions have been investigated for 
various classes of impulsive evolution equations, linear and semilinear problems by using some results from the theory of strongly continuous semigroups, the Fredholm alternative type
theorems and the fixed point theorems.

Before proceeding any further, we would like to observe that the notion of $( w,\mathbb{T}) $-periodicity is very general, although it might not seem at first sight. For example, if
$(T(t))_{t\geq 0}$ is a strongly continuous non-degenerate semigroup generated by a closed linear operator $A$ in $X$, then the unique mild solution of the abstract first-order Cauchy problem $u^{\prime}(t)=Au(t),$ $t\geq 0,$ equipped with the initial condition $u(0)=x,$ is given by $u(t):=T(t)x,$ $t\geq 0;$ see \cite{a43}, \cite{FKP} and references cited therein for further information concerning strongly continuous semigroups. It is clear that this solution has the property that for each $\omega>0$ we have
$u(t+\omega)=T(\omega)u(t),$ $t\geq 0$ and $u(\cdot)$ is therefore $(\omega, T(\omega))$-periodic for each $\omega>0$; needless to say that, if $(T(t))_{t\geq 0}$ can be extended to a strongly continuous group in $X,$ then the operator $T(\omega)$ is a linear isomorphism. This particularly holds if $A$ is a complex matrix of format $n\times n;$ then the unique solution of the system of ordinary differential equations
$u^{\prime}(t)=Au(t),$ $t\geq 0;$ $u(0)=x,$ given by $u(t):=T(t)x,$ $t\geq 0,$ is $(\omega, T(\omega))$-periodic for each real number $\omega>0.$ 
 
Further on, in \cite{marko-manuel-ap}, A. Ch\'avez, K. Khalil, M. Kosti\' c and M. Pinto have analyzed
various classes of almost periodic functions of form $F : I \times X\rightarrow Y,$ where $(Y,\|\cdot \|_{Y})$ is a complex Banach spaces and $\emptyset \neq  I \subseteq {\mathbb R}^{n}$;
the multi-dimensional $c$-almost periodic type functions have recently been investigated in \cite{c1-fil}, while the multi-dimensional $(\omega,c)$-almost periodic type functions have recently been investigated in \cite{nds-2021}; here, $c\in {\mathbb C}\setminus \{0\}$. For further information concerning multi-dimensional ($c$-)almost periodic type functions, multi-dimensional almost automorphic type functions and their applications, we refer the reader to the forthcoming research monograph \cite{nova-selected} by M. Kosti\' c (see also \cite{marko-manuel-ap}-\cite{marko-manuel-aa}).

For multi-periodic solutions of various classes of ordinary differential equations and partial differential equations, we also refer the reader to \cite{Berzha, Berzha1, A-FONDA, Kenzh, Kulz1, Kulz2, mingareli, multi-sarta1, multi-sarta2,
umbet1970, umbet1972, Umbet}.
Especially, we would like to mention the investigations of
G. Nadin \cite{g-nadin1}-\cite{g-nadin3}
concerning the space-time periodic reaction-diffusion equations, G. Nadin-L.Rossi \cite{g-nadin4} concerning transition waves for Fisher–KPP equations, L. Rossi \cite{luca-rossi} concerning Liouville type results for almost periodic type linear
operators,
the investigation of B. Scarpellini \cite{bruno-scar} concerning
the space almost periodic solutions of reaction-diffusion equations, and the recent investigation of R. Xie, Z. Xia, J. Liu \cite{xie-xia} concerning the quasi-periodic limit functions, $(\omega_{1},\omega_{2})$-(quasi)-periodic
limit functions and their applications, given only in the two-dimensional setting.

The main aim of this paper is to continue some of the above-mentioned research studies by investigating various classes of multi-dimensional $\rho$-almost periodic type functions $F : I \times X \rightarrow Y$ and multi-dimensional $(\omega,\rho)$-almost periodic type functions $F : I \times X \rightarrow Y$, where $n\in {\mathbb N},$
$\emptyset \neq I \subseteq {\mathbb R}^{n},$ $X$ and $Y$ are complex Banach spaces and $\rho$ is a binary relation on $Y.$ As mentioned in the abstract, our notion is completely new even in the one-dimensional setting. The introduced notion seems to be very general and we feel it is our duty to say that we have not been able to perceive with forethought all related problems and questions concerning multi-dimensional $\rho$-almost periodic type functions here (even the one-dimensional setting requires further analyses; see e.g., \cite{nova-mono}). It is our strong belief that this is only the pioneering paper in this direction which will motivate other authors to further explore this interesting topic in the near future. 

The organization and main ideas of this paper can be briefly described as follows. After collecting some necessary definitions and results about
the general binary relations, multivalued linear operators,
affine-periodic solutions and pseudo affine-periodic
solutions for various classes of systems of ordinary differential equations (Subsection \ref{affine}), we analyze multi-dimensional 
Bohr $({\mathcal B},I',\rho)$-almost periodic functions and 
multi-dimensional 
$({\mathcal B},I',\rho)$-uniformly recurrent type functions in Section \ref{ce-alm-per}. The basic notion is introduced in Definition \ref{nafaks123456789012345} and further analyzed in Proposition \ref{prcko}, Corollary \ref{rtanj}, Proposition \ref{adding}, Remark \ref{tomka} and Proposition \ref{bounded-pazi}. After that, we justify the introduction of notion in Example \ref{pripaz}-Example \ref{wseccha}.
The main structural characterizations of Bohr $({\mathcal B},I',\rho)$-almost periodic functions ($({\mathcal B},I',\sigma)$-uniformly recurrent functions) are given in Theorem \ref{lojalni}.
After that, we clarify Proposition \ref{superstebagT} and the supremum formula for $({\mathcal B},I',\rho)$-uniformly recurrent type functions
in Proposition \ref{deb}. The convolution invariance of Bohr $({\mathcal B},I',\rho)$-almost periodic ($({\mathcal B},I',\rho)$-uniformly recurrent) functions is investigated in Theorem \ref{milenko}; after that, in Theorem \ref{mkmk}, we analyze the invariance of Bohr $({\mathcal B},I',\rho)$-almost periodicity and $({\mathcal B},I',\sigma)$-uniform recurrence under the actions of the infinite convolution product 
\begin{align}\label{pikford}
{\bf t} \mapsto
\mapsto F({\bf t}):=\int^{t_{1}}_{-\infty}\int^{t_{2}}_{-\infty}\cdot \cdot \cdot \int^{t_{n}}_{-\infty} R({\bf t}-{\bf s})f({\bf s})\, d{\bf s},\quad {\bf t}\in {\mathbb R}^{n}.
\end{align}
We also clarify some obvious composition principles necessary for studying of the semilinear Cauchy problems. In Subsection \ref{nova}, we further analyze the notion of $T$-almost periodicity in the case that $Y={\mathbb C}^{n}$ is finite-dimensional and $T=A$ is a complex matrix of format $k\times k$, where $k\in {\mathbb N}.$ We provide many interesting results and observations in the case that $A$ is singular or non-singular; sometimes the geometry of region $I\subseteq {\mathbb R}^{n}$ is crucial for the validity of these results.
In the case that the matrix $A$ is singular, we construct an example of an $A$-almost periodic function $F : [0,\infty) \rightarrow {\mathbb C}^{2}$ which is not almost periodic.

Subsection \ref{mono-mono} thoroughly investigates ${\mathbb D}$-asymptotically Bohr $({\mathcal B},I',\rho)$-almost periodic type functions. In this subsection, we prove an extension type theorems for multi-dimensional $T$-almost periodic functions, where $T\in L(Y)$ is a linear isomorphism and the region $I$ satisfies certain properties (Theorem \ref{lenny-jassonceT}). The main aim of Section \ref{krizni-stab} is to analyze the corresponding problems for $({\bf \omega},\rho)$-periodic functions and $({\bf \omega}_{j},\rho_{j})_{j\in {\mathbb N}_{n}}$-periodic functions. The existence of mean value of multi-dimensional $\rho$-almost periodic functions is analyzed in Corollary \ref{rtanj1}, Corollary \ref{why} and an interesting problem concerning this issue is proposed at the end of the second section.
In the final section of paper, we provide several illustrative applications to the abstract (degenerate) Volterra integro-differential equations in Banach spaces.

Before explaining the notation and terminology used in the paper, the authors would like to express their sincere gratitude to Professor G. M. N'Gu\' er\' ekata, who invited them to write and publish a paper for the special issue of Applicable Analysis dedicated to the memory of  Professor A. A. Pankov. The monograph ``Bounded and Almost Periodic Solutions of Nonlinear Operator Differential
Equations'' published by Kluwer Acad. Publ. in 1990 can serve as an excellent source for the readers interested to acquire a basic knowledge about Bohr compactifications and almost periodic functions on topological groups; see also the research monographs \cite{pankov1}-\cite{pankov11} for some other scientific research interests of Professor A. A. Pankov.

\vspace{0.1cm}

\noindent {\bf Notation and terminology.} Suppose that $X,\ Y,\ Z$ and $ T$ are given non-empty sets. Let us recall that a binary relation between $X$ into $Y$
is any subset
$\rho \subseteq X \times Y.$ 
If $\rho \subseteq X\times Y$ and $\sigma \subseteq Z\times T$ with $Y \cap Z \neq \emptyset,$ then
we define $\rho^{-1} \subseteq Y\times X$
and
$\sigma \cdot  \rho =\sigma \circ \rho \subseteq X\times T$ by
$
\rho^{-1}:=\{ (y,x)\in Y\times X : (x,y) \in \rho \}
$
and
$$
\sigma \circ \rho :=\bigl\{(x,t) \in X\times T : \exists y\in Y \cap Z\mbox{ such that }(x,y)\in \rho\mbox{ and }
(y,t)\in \sigma \bigr\},
$$
respectively. As is well known, the domain and range of $\rho$ are defined by $D(\rho):=\{x\in X :
\exists y\in Y\mbox{ such that }(x,y)\in X\times Y \}$ and $R(\rho):=\{y\in Y :
\exists x\in X\mbox{ such that }(x,y)\in X\times Y\},$ respectively; $\rho (x):=\{y\in Y : (x,y)\in \rho\}$ ($x\in X$), $ x\ \rho \ y \Leftrightarrow (x,y)\in \rho .$
If $\rho$ is a binary relation on $X$ and $n\in {\mathbb N},$ then we define $\rho^{n}
$ inductively; $\rho^{-n}:=(\rho^{n})^{-1}$ and $\rho^{0}:=\Delta_{X}:=
\{(x,x) : x\in X\}.$ Set $\rho (X'):=\{y : y\in \rho(x)\mbox{ for some }x\in X'\}$ ($X'\subseteq X$) and
${\mathbb N}_{n}:=\{1,\cdot \cdot \cdot,n\}$ ($n\in {\mathbb N}$). For any set $A$ we define its power set $P(A):=\{ B \, |\,  B\subseteq A\}.$ An unbounded subset $A\subseteq {\mathbb N}$ is called syndetic if and only if there exists a strictly increasing sequence $(a_{n})$ of positive integers such that $A=\{ a_{n} : n\in {\mathbb N}\} $ and $\sup_{n\in {\mathbb N}}(a_{n+1}-a_{n})<+\infty.$ 

We assume henceforth that $(X,\| \cdot \|),$ $(Y, \|\cdot\|_Y)$ and $(Z, \|\cdot\|_Z)$ are three complex Banach spaces, $n\in {\mathbb N},$ ${\mathcal B}$ is a certain collection of subsets
of $X$ satisfying
that
for each $x\in X$ there exists $B\in {\mathcal B}$ such that $x\in B.$ By $L(X,Y)$ we denote the Banach space of all linear continuous functions from $X$ into $Y;$ $L(X)\equiv L(X,X).$  We will always use the principal branch of the exponential function to take the powers of complex numbers. 
A multivalued map (multimap) ${\mathcal A} : X \rightarrow P(Y)$ is said to be a multivalued
linear operator (MLO) if and only if the following holds:
\begin{itemize}
\item[(i)] $D({\mathcal A}) := \{x \in X : {\mathcal A}x \neq \emptyset\}$ is a linear subspace of $X$;
\item[(ii)] ${\mathcal A}x +{\mathcal A}y \subseteq {\mathcal A}(x + y),$ $x,\ y \in D({\mathcal A})$
and $\lambda {\mathcal A}x \subseteq {\mathcal A}(\lambda x),$ $\lambda \in {\mathbb C},$ $x \in D({\mathcal A}).$
\end{itemize}
If $X=Y,$ then we say that ${\mathcal A}$ is an MLO in $X.$ The sums, products and integer powers of MLOs are defined usually (see \cite[Section 1.2]{FKP} for more details about the subject). The kernel space $N({\mathcal A})$ of ${\mathcal A}$ is defined as the set of all $x\in D({\mathcal A})$ such that $0\in {\mathcal A}x.$ It is said that an MLO operator  ${\mathcal A} : X\rightarrow P(Y)$ is closed if and only if for any
sequences $(x_{n})$ in $D({\mathcal A})$ and $(y_{n})$ in $Y$ such that $y_{n}\in {\mathcal A}x_{n}$ for all $n\in {\mathbb N}$ we have that $\lim_{n \rightarrow \infty}x_{n}=x$ and
$\lim_{n \rightarrow \infty}y_{n}=y$ imply
$x\in D({\mathcal A})$ and $y\in {\mathcal A}x.$
We need the following lemma (see e.g., \cite[Theorem 1.2.3]{FKP}):

\begin{lem}\label{stana}
Suppose that ${\mathcal A} : X\rightarrow P(Y)$ is a closed \emph{MLO}, $\Omega$ is a locally compact, separable metric space\index{space!separable metric}
and $\mu$ is a locally finite
Borel measure\index{measure!locally finite Borel} defined on $\Omega.$
Let $f : \Omega \rightarrow X$ and $g : \Omega \rightarrow Y$ be $\mu$-integrable, and let $g(x)\in {\mathcal A}f(x),$ $x\in \Omega.$ Then $\int_{\Omega}f\, d\mu \in D({\mathcal A})$ and $\int_{\Omega}g\, d\mu\in {\mathcal A}\int_{\Omega}f\, d\mu.$
\end{lem}

If ${\bf t_{0}}\in {\mathbb R}^{n}$ and $\epsilon>0$, then we set $B({\bf t}_{0},\epsilon):=\{{\bf t } \in {\mathbb R}^{n} : |{\bf t}-{\bf t_{0}}| \leq \epsilon\},$
where $|\cdot|$ denotes the Euclidean norm in ${\mathbb R}^{n}.$ Set 
$I_{M}:=\{{\bf t}\in I : |{\bf t}| \geq M\}$ ($I\subseteq {\mathbb R}^{n};$ $M>0$).  Further on,
by $(e_{1},e_{2},\cdot \cdot \cdot,e_{n})$ we denote the standard basis of ${\mathbb R}^{n};$ $\langle \cdot, \cdot \rangle$ denotes the usual inner product in ${\mathbb R}^{n}.$

We will use the following definition from \cite{marko-manuel-ap}:

\begin{defn}\label{kompleks12345}
Suppose that 
${\mathbb D} \subseteq I \subseteq {\mathbb R}^{n}$ and the set ${\mathbb D}$  is unbounded. By $C_{0,{\mathbb D},{\mathcal B}}(I \times X :Y)$ we denote the vector space consisting of all continuous functions $Q : I \times X \rightarrow Y$ such that, for every $B\in {\mathcal B},$ we have $\lim_{t\in {\mathbb D},|t|\rightarrow +\infty}Q({\bf t};x)=0,$ uniformly for $x\in B.$
If $X=\{0\},$ then we abbreviate $C_{0,{\mathbb D},{\mathcal B}}(I \times X :Y)$ to 
$C_{0,{\mathbb D},{\mathcal B}}(I :Y);$ furthermore, if ${\mathbb D}=I,$ then we omit the term ``${\mathbb D}$'' from the notation.
\end{defn}

\subsection{Affine-periodic solutions and pseudo affine-periodic
solutions for various classes of systems of ordinary differential equations}\label{affine} In a great number of
recent research studies, the notions of affine-periodicity and pseudo affine-periodicity play
an incredible role in the qualitative analysis of solutions for various classes of systems of ordinary differential equations, systems of functional
differential equations and systems of Newtonian equations of motion
with friction; see e.g.,
\cite{affine-chang, affine-cheng, affine-li, affine-li1, affine-meng, affine-wang, affine-wang1, affine-xia, affine-zhang} for some results obtained by Chinese mathematicians in this direction. In this subsection, we will only insribe the main ideas of research studies carried out by
X. Chang, Y. Li in \cite{affine-chang} and Y. Li, H. Wang, X. Yang in \cite{affine-li1}.
By a $(Q, T)$ affine-periodic function $x: {\mathbb R} \rightarrow {\mathbb R}^{n}$
we mean any continuous function $x(\cdot)$ for which
$$
x(t+T)=Qx(t),\quad t\in {\mathbb R},
$$
where
$Q$ is a regular matrix of format $n\times n$ and $T>0.$ Some qualitative properties of $x(\cdot)$ like periodicity, subharmonicity, or quasi-periodicity are induced by the corresponding qualitative properties of matrix $ Q;$ the most important subclasses of regular matrices for $Q$ are
power identity matrices, i.e., those matrices for which we have $ Q^{k} = I $ for some integer $k\in {\mathbb Z},$ or orthogonal
matrices belonging to the group $O(n)$; these subclasses are important for modeling certain real phenomena describing rotation motions in body from mechanics.

In \cite{affine-chang},
X. Chang and Y. Li
have investigated the rotating periodic solutions of second-order dissipative dynamical systems. More precisely, the authors
have considered the following dissipative dynamical
system:
$$
u^{\prime \prime}+cu^{\prime}+\nabla g(u)+h(u)=e(t),\quad t\in {\mathbb R},
$$
where $c > 0$ is a constant, $g(u)=g(|u|),$ $h \in C({\mathbb R}^{n} : {\mathbb R}^{n}),$
$h(u)=Qh(Q^{-1}u)$ for some orthogonal matrix $Q\in O(n)$ and $e\in C({\mathbb R} : {\mathbb R}^{n})$
satisfies $e(t+T)=Qe(t)$ for all $t\in {\mathbb R}.$ It has been shown that the above equation admits a solution of the form $u(t + T) = Qu(t), $ $t\in {\mathbb R},$ which is usually called rotating periodic solution.

In \cite{affine-li1}, Y. Li, H. Wang and X. Yang
have analyzed Fink's conjecture on affine-periodic solutions and Levinson's conjecture to Newtonian systems. The authors have analyzed the following system of ordinary differential equations
$$
x^{\prime}(t)=f(t,x(t)),\quad t\in {\mathbb R},
$$
where $f \in  C({\mathbb R} \times {\mathbb R}^{n} : {\mathbb R}^{n}),$ $f(t,x) \equiv Qf(t,Q^{-1}x),$ the following system of functional-differential equations
$$
x^{\prime}(t)=F\bigl(t,x_{t}\bigr),\quad t\in {\mathbb R},
$$
where $x_{t}(s)=x(t+s)$ for $s\in [-r,0]$ and fixed $r>0,$ $F: {\mathbb R} \times C\rightarrow {\mathbb R}^{n}$ is continuous with $C$ being the Banach space of continuous functions $C([-r,0] : {\mathbb R}^{n})$ equipped with the sup-norm and
$F(t,\varphi) \equiv QF(t,Q^{-1}\varphi),$ and the following system of Newtonian equations of motion
with friction
$$
x^{\prime \prime}+A(t,x)u^{\prime}+\nabla V(x)+h(u)=e(t),\quad t\in {\mathbb R},
$$
where $A: {\mathbb R}\times {\mathbb R}^{m} \rightarrow {\mathbb R}^{m,m},$ $V : {\mathbb R}^{m}\rightarrow {\mathbb R}$ and $e :{\mathbb R} \rightarrow {\mathbb R}^{m}$ are continuous, $A(t,x)$ satisfies the local Lipschitz condition with respect to the variable $x$, $V(\cdot)$ is continuously differentiable and
 $A(t+T,x)y \equiv  QA(t,Q^{-1}x)Q^{-1}y,$ $\nabla V(x) \equiv Q\nabla V(Q^{-1}x),$ $e(t+T)\equiv Qe(t).$ Following the authors, such a vectorial equation is called $(Q, T)$ affine-periodic ordinary
differential equation, a $(Q, T)$ affine-periodic functional differential equation, or a
$(Q, T)$ affine-periodic Newtonian equation, respectively. Practically, the authors have essentially verified Levinson's conjecture for Newtonian systems with friction and proposed the problem of existence of
a $(Q, T)$ affine-periodic solution for a Newtonian system with friction.

Concerning $(Q, T)$ affine-periodic solutions of ordinary differential equations, we would like to note that we will consider here a singularly perturbed symmetric ODE with a symmetric
heteroclinic cycle connecting hyperbolic equilibria (see Subsection \ref{sec3}). We will analyze the accumulation of
$(Q,T)$ affine-periodic solutions 
on the heteroclinic cycles and the phenomenon of 
heteroclinic/homoclinic period blow-up. See also \cite{aizo, forced} and references cited therein.

\section{Multi-dimensional $\rho$-almost periodic type functions}\label{ce-alm-per}

In \cite{nova-mse}, we have recently introduced and analyzed the following notion with $\rho =c{\rm I},$ where ${\rm I}$ denotes the identity operator on $Y:$

\begin{defn}\label{nafaks123456789012345}
Suppose that $\emptyset  \neq I' \subseteq {\mathbb R}^{n},$ $\emptyset  \neq I \subseteq {\mathbb R}^{n},$ $F : I \times X \rightarrow Y$ is a continuous function, $\rho$ is a binary relation on $Y$ and $I +I' \subseteq I.$ Then we say that:
\begin{itemize}
\item[(i)]\index{function!Bohr $({\mathcal B},I',\rho)$-almost periodic}
$F(\cdot;\cdot)$ is Bohr $({\mathcal B},I',\rho)$-almost periodic if and only if for every $B\in {\mathcal B}$ and $\epsilon>0$
there exists $l>0$ such that for each ${\bf t}_{0} \in I'$ there exists ${\bf \tau} \in B({\bf t}_{0},l) \cap I'$ such that, for every ${\bf t}\in I$ and $x\in B,$ there exists an element $y_{{\bf t};x}\in \rho (F({\bf t};x))$ such that
\begin{align}\label{oblak}
\bigl\|F({\bf t}+{\bf \tau};x)-y_{{\bf t};x}\bigr\|_{Y} \leq \epsilon .
\end{align}
\item[(ii)] \index{function!$({\mathcal B},I',\rho)$-uniformly recurrent}
$F(\cdot;\cdot)$ is $({\mathcal B},I',\rho)$-uniformly recurrent if and only if for every $B\in {\mathcal B}$ 
there exists a sequence $({\bf \tau}_{k})$ in $I'$ such that $\lim_{k\rightarrow +\infty} |{\bf \tau}_{k}|=+\infty$ and that, for every ${\bf t}\in I$ and $x\in B,$ there exists an element $y_{{\bf t};x}\in \rho (F({\bf t};x))$ such that
\begin{align}\label{oblak1}
\lim_{k\rightarrow +\infty}\sup_{{\bf t}\in I;x\in B} \bigl\|F({\bf t}+{\bf \tau}_{k};x)-y_{{\bf t};x}\bigr\|_{Y} =0.
\end{align}
\end{itemize}
\end{defn}

It is clear that the  Bohr $({\mathcal B},I',\rho)$-almost periodicity of $F(\cdot;\cdot)$ implies the $({\mathcal B},I',\rho)$-uniform recurrence of $F(\cdot;\cdot)$; the converse statement is not true in general (\cite{nova-selected}).
In the case that $\rho=T :  Y \rightarrow Y$ is a single-valued function (not necessarily linear or continuous), then we obtain the most important case for our further investigations, when the function $F(\cdot;\cdot)$ is $({\mathcal B},I',T)$-almost periodic, resp. $({\mathcal B},I',T)$-uniformly recurrent. In the case that $X=\{0\}$ ($I'=I$), we omit the term ``${\mathcal B}$'' (``$I'$'') from the notation; furthermore, if $T=c{\rm I}$ for some complex number $c\in {\mathbb C} \setminus \{0\}$, then we also say that the function $F(\cdot;\cdot)$ is $({\mathcal B},I',c)$-almost periodic, resp. $({\mathcal B},I',c)$-uniformly recurrent.

In \cite[Proposition 2.6, Corollary 2.10, Proposition 2.11]{c1}, 
we have considered the question whether a given one-dimensional $c$-almost periodic function ($c$-uniformly recurrent function), where $c\in S_{1}\equiv \{ z\in {\mathbb C} : |z|=1\},$ is almost periodic (uniformly recurrent). Concerning this question for general binary relations in the multi-dimensional setting, we will state and prove the following result (see also Proposition \ref{tomqa} and Example \ref{tomqqa} below for the case that the space $Y$ is finite-dimensional and the equality $\tau +I=I$ is not satisfied for some points ${\bf \tau}\in I'$):

\begin{prop}\label{prcko}
Suppose that $\emptyset  \neq I' \subseteq {\mathbb R}^{n},$ $\emptyset  \neq I \subseteq {\mathbb R}^{n},$  $I +I' \subseteq I$, and the function $F : I \times X \rightarrow Y$ is Bohr $({\mathcal B},I',\rho)$-almost periodic ($({\mathcal B},I',\rho)$-uniformly recurrent), where $\rho$ is a binary relation on $Y$ satisfying $R(F)\subseteq D(\rho)$ and $\rho(y)$ is a singleton for any $y\in R(F).$ If for each ${\bf \tau}\in I'$ we have $\tau +I=I,$ then $I+(I'-I')\subseteq I$ and the function $F(\cdot;\cdot)$ is Bohr $({\mathcal B},I'-I',{\rm I})$-almost periodic ($({\mathcal B},I'-I',{\rm I})$-uniformly recurrent).
\end{prop}

\begin{proof}
We will consider only Bohr $({\mathcal B},I',\rho)$-almost periodic functions. Let $\tau\in I'-I'$ be given. Then there exist points $\tau_{1},\ \tau_{2}\in I'$ such that $\tau=\tau_{1}-\tau_{2};$ if ${\bf t}\in I,$ then the prescribed assumption $\tau_{2}+I=I$ implies the existence of a point ${\bf t}'\in I$ such that ${\bf t}=\tau_{2}+{\bf t}'.$ Hence, ${\bf t}+\tau=\tau_{1}+{\bf t}'\in I$ and therefore $I+(I'-I')\subseteq I$, as claimed. Further on, let $\epsilon>0$ and $B\in {\mathcal B}$ be given. Then there exists $l>0$ such that for each ${\bf t}_{0}^{1},\  {\bf t}_{0}^{2} \in I'$ there exist two points ${\bf \tau}_{1} \in B({\bf t}_{0}^{1},l) \cap I'$ and
${\bf \tau}_{2} \in B({\bf t}_{0}^{2},l) \cap I'$
such that, for every ${\bf t}\in I$ and $x\in B,$ we have
\begin{align*}
\bigl\|F\bigl({\bf t}+{\bf \tau}_{1};x\bigr)-\rho(F({\bf t};x))\bigr\|_{Y} \leq \epsilon/2 \ \ \mbox{ and }\ \ \bigl\|F\bigl({\bf t}+{\bf \tau}_{2};x\bigr)-\rho(F({\bf t};x))\bigr\|_{Y} \leq \epsilon/2.
\end{align*}
This implies
\begin{align*}
\bigl\|F\bigl({\bf t}+{\bf \tau}_{1};x\bigr)-F\bigl({\bf t}+{\bf \tau}_{2};x\bigr)\bigr\|_{Y} \leq \epsilon,\quad {\bf t}\in I,\ x\in B,
\end{align*}
i.e.,
\begin{align*}
\bigl\|F\bigl({\bf v}+\bigl[\tau_{2}-{\bf \tau}_{1}\bigr];x\bigr)-F\bigl({\bf v};x\bigr)\bigr\|_{Y} \leq \epsilon,\quad {\bf v}\in \tau_{1}+I=I,\ x\in B.
\end{align*}
Since $\tau_{2}-\tau_{1}\in B({\bf t}_{0}^{2}-{\bf t}_{0}^{1},2l) \cap (I'-I'),$ this simply implies the required conclusion by definition.
\end{proof}

\begin{cor}\label{rtanj}
Suppose that $\emptyset  \neq I' \subseteq {\mathbb R}^{n},$ and the function $F : {\mathbb R}^{n} \times X \rightarrow Y$ is Bohr $({\mathcal B},I',\rho)$-almost periodic ($({\mathcal B},I',\rho)$-uniformly recurrent), where $\rho$ is a binary relation on $Y$ satisfying $R(F)\subseteq D(\rho)$ and $\rho(y)$ is a singleton for any $y\in R(F).$ Then the function $F(\cdot;\cdot)$ is Bohr $({\mathcal B},I'-I',{\rm I})$-almost periodic ($({\mathcal B},I'-I',{\rm I})$-uniformly recurrent).
\end{cor}

In the case that $I={\mathbb R},$ then Corollary \ref{rtanj} enables one to see that any $c$-uniformly recurrent function $F : {\mathbb R} \rightarrow Y$, where $c\in {\mathbb C}\setminus \{0\}$, is uniformly recurrent (see \cite[Definition 2.3]{c1}). Strictly speaking, this is a new result which is not clarified in \cite{c1}: more precisely, in \cite[Corollary 2.10, Proposition 2.11(ii)]{c1}, we have proved that any bounded $c$-uniformly recurrent function $F : I \rightarrow Y$, where $c\in {\mathbb C}\setminus \{0\}$ satisfies $|c|=1$ and $I={\mathbb R}$ or $I=[0,\infty),$ is uniformly recurrent (if $|c|\neq 1,$ then the unique $c$-uniformly recurrent function $F : I \rightarrow Y$ is the zero function; see \cite[Proposition 2.6]{c1}).

We also have the following important corollary of Proposition \ref{prcko}:

\begin{cor}\label{rtanj1}
Suppose that $\emptyset  \neq I' \subseteq {\mathbb R}^{n},$ $I'-I'={\mathbb R}^{n}$ and the function $F : {\mathbb R}^{n} \rightarrow Y$ is Bohr $(I',\rho)$-almost periodic, where $\rho$ is a binary relation on $Y$ satisfying $R(F)\subseteq D(\rho)$ and $\rho(y)$ is a singleton for any $y\in R(F).$ Then the function $F(\cdot)$ has the mean value
\begin{align}\label{idiote}
{\mathcal M}_{\lambda}(F):=\lim_{T\rightarrow +\infty}\frac{1}{(2T)^{n}}\int_{K_{T}}e^{-i\langle \lambda, {\bf t}\rangle} F({\bf t})\, d{\bf t}=\lim_{T\rightarrow +\infty}\frac{1}{T^{n}}\int_{L_{T}}e^{-i\langle \lambda, {\bf t}\rangle} F({\bf t})\, d{\bf t}
\end{align}
for any $\lambda \in {\mathbb R}^{n}$, and the set of all points $\lambda \in {\mathbb R}^{n}$ for which ${\mathcal M}_{\lambda}(F)\neq 0$ is at most countable; here, for every $T>0,$
$
K_{T}:=\{ (t_{1},t_{2},\cdot \cdot \cdot, t_{n})\in {\mathbb R}^{n} : |t_{i}| \leq T,\ i\in {\mathbb N}_{n}\}$ and $L_{T}
:=\{(t_{1},t_{2},\cdot \cdot \cdot, t_{n})\in [0,\infty)^{n} : t_{i} \leq T,\ i\in {\mathbb N}_{n}\}.
$
\end{cor}

For our later purposes, it would be very useful to formulate the following result:

\begin{prop}\label{adding}
Suppose that $\emptyset  \neq I' \subseteq {\mathbb R}^{n},$ $\emptyset  \neq I \subseteq {\mathbb R}^{n},$ $I +I' \subseteq I,$ $\rho =A$ is a linear non-injective operator on $Y$, and $G : I \times X \rightarrow Y$ is a Bohr $({\mathcal B},I',A)$-almost periodic ($({\mathcal B},I',A)$-uniformly recurrent) function. Let $Q : I \times X \rightarrow N(A)$ be any continuous function satisfying that for each $B\in {\mathcal B}$ we have $\lim_{|{\bf t}|\rightarrow +\infty, {\bf t}\in I+I'}Q({\bf t};x)=0,$ uniformly for $x\in B.$ Suppose, further,  that the following condition holds:
\begin{itemize}
\item[(D)] For every ${\bf t}_{0} \in I',$ for every $l>0$ and for every $l'\geq 2l,$ there exists ${\bf t}_{0}'\in I'$ such that $B({\bf t}_{0}',l) \subseteq B({\bf t}_{0},2l')$ and $|{\bf t}+\tau|\geq l'$ for all ${\bf t}\in I$ and $\tau \in B({\bf t}_{0}',l).$
\end{itemize}
Then the function $F=G+Q : I \times X \rightarrow Y$ is likewise Bohr $({\mathcal B},I',A)$-almost periodic ($({\mathcal B},I',A)$-uniformly recurrent). 
\end{prop}

\begin{proof}
We will consider only Bohr $({\mathcal B},I',A)$-almost periodic functions.
Let $B\in {\mathcal B}$ and $\epsilon>0$ be fixed. 
Then there exists $l>0$ such that for each ${\bf t}_{0} \in I'$ there exists ${\bf \tau} \in B({\bf t}_{0},l) \cap I'$ such that, for every ${\bf t}\in I$ and $x\in B,$ the element $y_{{\bf t};x}= A (G({\bf t};x))$ satisfies \eqref{oblak} with the function $F(\cdot;\cdot)$ and the number $\epsilon$ replaced respectively with the function $G(\cdot;\cdot)$ and the number $\epsilon/2.$
Our assumption implies the existence of a finite real number $l_{0}(\epsilon)>0$ such that for each ${\bf t}\in I+I'$ with $|{\bf t}|\geq l_{0}(\epsilon)$ and $x\in B$ we have
$\| Q({\bf t};x)\|_{Y}\leq \epsilon/2.$ Take now $l':=2\max(l,l_{0}(\epsilon)).$ Then the requirements of Definition \ref{nafaks123456789012345} are satisfied with the number $l$ replaced therein with the number $2l'.$ In actual fact, let a point ${\bf t}_{0}\in I'$ be given. Then we can find a point ${\bf t}_{0}'\in I'$ in accordance with condition (D). Since $G(\cdot;\cdot)$ is Bohr $({\mathcal B},I',A)$-almost periodic, we have the existence of a point $\tau \in B({\bf t}_{0}',l) \cap I'$ such that $\| G({\bf t}+\tau;x)-AG({\bf t};x)\|_{Y}\leq \epsilon/2$ for all ${\bf t}\in I$ and $x\in B.$ Due to (D), we have $\tau \in B({\bf t}_{0},2l') \cap I'$ and $|{\bf t}+\tau|\geq l'$ for all ${\bf t}\in I$; hence, $\| Q({\bf t}+\tau;x)\|_{Y}\leq \epsilon/2$ for all ${\bf t}\in I$ and $x\in B$, which simply implies the required statement since $R(Q)\subseteq N(A)$.
\end{proof}

\begin{rem}\label{tomka}
Condition (D) is valid in the case that $I=[0,\infty)$ and $I'=(0,\infty).$ Then, for every $t_{0}>0,$ $l>0$ and $l'\geq 2l,$ we can take $t_{0}':=t_{0}+2l'-l.$
\end{rem}

It is worth noting that the statements of \cite[Proposition 2.16, Proposition 2.21]{c1} can be reformulated for $T$-almost periodicity, where $T\in L(Y)$ is a linear isomorphism since, in this case, the estimate \eqref{oblak} implies $\| T^{-1}F({\bf t}+\tau;x)-F({\bf t};x)\|_{Y} \leq \epsilon \cdot \| T^{-1}\|_{L(Y)}$ for all ${\bf t}\in I$ and $x\in B.$ The proofs are very similar to those in which $T={\rm I}$ and therefore omitted:

\begin{prop}\label{bounded-pazi}
Suppose that $\rho=T\in L(Y)$ is a linear isomorphism.
\begin{itemize}
\item[(i)]
Suppose that $\emptyset  \neq I \subseteq {\mathbb R}^{n},$ $I +I \subseteq I,$ $I$ is closed,
$F : I \times X \rightarrow Y$ is Bohr $({\mathcal B},T)$-almost periodic and ${\mathcal B}$ is any family of compact subsets of $X.$ If
\begin{align*}
(\forall l>0) \, (\exists {\bf t_{0}}\in I)\, (\exists k>0) &\, (\forall {\bf t} \in I)(\exists {\bf t_{0}'}\in I)\,
\\ & (\forall {\bf t_{0}''}\in B({\bf t_{0}'},l) \cap I)\, {\bf t}- {\bf t_{0}''} \in B({\bf t_{0}},kl) \cap I,
\end{align*}
then for each $B\in {\mathcal B}$ we have 
that the set $\{ F({\bf t}; x) : {\bf t} \in I,\ x\in B\}$ is relatively compact in $Y;$
in particular,
$\sup_{{\bf t}\in I;x\in B}\|F({\bf t}; x)\|_{Y}<\infty.$
\item[(ii)] Suppose that $\emptyset  \neq I \subseteq {\mathbb R}^{n},$ $I +I \subseteq I,$ $I$ is closed and $F : I \times X \rightarrow Y$ is Bohr $({\mathcal B},T)$-almost periodic, where ${\mathcal B}$ is a family consisting of some compact subsets of $X.$ If the following condition holds
\begin{align*}
(\exists {\bf t_{0}}\in I)\, (\forall \epsilon>0)
(\forall l>0) \, (\exists l'>0) &\, (\forall {\bf t'},\ {\bf t''} \in I) \\ & B({\bf t_{0}},l) \cap I \subseteq B({\bf t_{0}-t'},l') \cap  B({\bf t_{0}-t''},l') ,
\end{align*}
then for each $B\in {\mathcal B}$ the function $F(\cdot;\cdot)$ is uniformly continuous on $I\times B.$ 
\end{itemize}
\end{prop}

We continue by providing some illustrative examples:

\begin{example}\label{pripaz}
There is no need to say that the linearity of operator $T$ is crucial sometimes; suppose, for simplicity, that $\rho=T\in L(Y),$ $I=I'=[0,\infty)$ or $I=I'={\mathbb R},$ and $X=\{0\}$ (see also \cite[Proposition 2.9]{c1} and \cite[Proposition 2.13]{nova-mse}). 
For each $t\in I,$ $\tau \in I$ and $l\in {\mathbb{N}},$ we have \newline
\begin{align*}
F\bigl(t +l\tau \bigr)-T^{l}F(t) 
=\sum_{j=0}^{l-1}T^{j}\Bigl[F\bigl(t+(l-j)\tau \bigr)-TF\bigl(t+(l-j-1)\tau \bigr)\Bigr].
\end{align*}
Hence, 
\begin{equation*}
\Bigl\|F\bigl(\cdot +l\tau \bigr)-T^{l}F(\cdot )\Bigr\|_{Y}\leq l
\Bigl\|F\bigl(\cdot +\tau \bigr)-TF(\cdot )\Bigr\|_{Y},
\end{equation*}
which implies that the function
$F(\cdot)$ is $T^{l}$-almost periodic ($T^{l}$-uniformly recurrent), provided that $F(\cdot)$ is $T$-almost periodic ($T$-uniformly recurrent). In particular, the function $F(\cdot)$ is almost periodic (uniformly recurrent), provided that $F(\cdot)$ is $T$-almost periodic ($T$-uniformly recurrent) and there exists a positive integer $l\in {\mathbb N}$ such that $T^{l}={\rm I}.$
\end{example}

\begin{example}\label{seccha}
The following example is taken from \cite{nova-selected}, where we have considered case $c=1,$ only; this example shows that it is slightly redundant to assume that
$I'\subseteq I$ in Definition \ref{nafaks123456789012345}.

Suppose that $L>0$ is a fixed real number as well as that the function $t\mapsto (f(t),g(t)),$ $t\in {\mathbb R}$ is $c$-almost periodic. Set $I:=\{(x,y) \in {\mathbb R}^{2} : |x-y| \geq L\},$ $I':=\{(\tau,\tau) : \tau\in {\mathbb R}\}$ and
$$
u(x,y):=\frac{f(x)+g(y)}{x-y},\quad (x,y)\in I.
$$
Then $I+I'\subseteq I$ but $I'$ is not a subset of $I.$ Furthermore, if $\epsilon>0$ is given and $\tau >0$ is a common $(\epsilon ,c)$-period of the functions $f(\cdot)$ and $g(\cdot)$, then we have:
\begin{align*}
\|u(x+\tau,y+\tau)-cu(x,y)\| &\leq \frac{\|f(x+\tau)-cf(x)\|+\|g(y+\tau)-cg(y)\|}{|x-y|}
\\& \leq 2\epsilon/L,\quad (x,y)\in I.
\end{align*}
This implies that the function $u(\cdot,\cdot)$ is Bohr $(I',c)$-almost periodic. Observe, finally, that under some regularity conditions on the functions $f(\cdot)$ and $g(\cdot),$ the function $u(\cdot,\cdot)$ is a solution of the partial differential equation
$$
u_{xy}-\frac{u_{x}}{x-y}+\frac{u_{y}}{x-y}=0.
$$
\end{example}

\begin{example}\label{wseccha}
The conclusions established in \cite[Example 2.5, Example 2.12]{nova-mse} (cf. also \cite[Example 2.13, Example 2.15]{marko-manuel-ap}) for $c$-almost periodicity can be formulated for $T$-almost periodicity, where $T\in L(Y),$ providing certain additional assumptions. For simplicity, let us consider the
concrete situation  
of \cite[Example 2.5(i)]{nova-mse}. Suppose that $F_{j} : X \rightarrow Y$ is a continuous function, for each $B\in {\mathcal B}$ we have $\sup_{x\in B}\| F_{j}(x) \|_{Y}<\infty,$ 
and the $Z$-valued mapping $t\mapsto (\int_{0}^{t}f_{1}(s)\, ds,\cdot \cdot \cdot, \int_{0}^{t}f_{n}(s)\, ds),$  $t\geq 0$ is bounded and $T$-almost periodic ($1\leq j \leq n$). Suppose, further, that the multiplication
$\cdot : Z\times Y \rightarrow Y$ is defined and satisfies:
\begin{itemize}
\item[(i)] there exists a finite real constant $c>0$ such that 
$\| zy\|_{Y} \leq c\|z\|_{Z}\|y\|_{Y}$ for all $z\in Z$ and $y\in Y;$
\item[(ii)] $z_{1}y-z_{2}y=(z_{1}-z_{2})y$ for all $z_{1},\ z_{2}\in Z$ and $y\in Y;$
\item[(iii)] $zy_{1}-zy_{2}=z(y_{1}-y_{2})$ for all $z\in Z$ and $y_{1},\ y_{2}\in Y$.
\end{itemize}
Set
\begin{align*}
F\bigl(t_{1},\cdot \cdot \cdot,t_{n+1}; x\bigr):=\sum_{j=1}^{n}\int_{t_{j}}^{t_{j+1}}f_{j}(s)\, ds \cdot F_{j}(x)\ \mbox{ for all }x\in X \mbox{ and } t_{j}\geq 0,\ 1\leq j\leq n.
\end{align*}
Arguing as in \cite[Example 2.13(i)]{marko-manuel-ap}, we may deduce that the mapping $F: [0,\infty)^{n+1} \times X \rightarrow Y$ is Bohr $({\mathcal B},T)$-almost periodic. 
\end{example}

The use of binary relation $\rho$ in Definition \ref{nafaks123456789012345} suggests a very general way of approaching to many known classes of almost periodic functions; before we go any further, we would like to note that this general approach has some obvious unpleasant  
consequences 
because, under the general requirements of Definition \ref{nafaks123456789012345}, any continuous function $F(\cdot;\cdot)$ is always Bohr $({\mathcal B},I',\rho)$-almost periodic provided that $R(F) \times R(F) \subseteq \rho$ (in particular, it is very redundant to assume any kind of boundedness of function $F(\cdot;\cdot)$ in Definition \ref{nafaks123456789012345}; see \cite[Proposition 2.2]{c1} and \cite[Proposition 2.8]{nds-2021} for some particular results obtained in this direction). Therefore, given two sets $\emptyset  \neq I' \subseteq {\mathbb R}^{n},$ $\emptyset  \neq I \subseteq {\mathbb R}^{n}$ satisfying $I +I' \subseteq I$
and a continuous function
$F : I \times X \rightarrow Y$, 
it is natural to introduce the following non-empty sets:
\begin{align}\label{pus}
 {\mathcal A}_{I',I,F}:=\bigl\{ \rho \subseteq Y\times Y \, ;\,  F({\cdot}; {\cdot})\mbox{ is Bohr }({\mathcal B},I',\rho)-\mbox{almost periodic} \bigr\}
\end{align}
and
\begin{align}\label{pus}
 {\mathcal B}_{I',I,F}:=\bigl\{ \rho \subseteq Y\times Y \, ; \, F({\cdot}; {\cdot})\mbox{ is }({\mathcal B},I',\rho)-\mbox{uniformly recurrent} \bigr\}.
\end{align}
Clearly, we have $\emptyset \neq {\mathcal A}_{I',I,F} \subseteq {\mathcal B}_{I',I,F};$ further on, the assumptions $\rho \in {\mathcal A}_{I',I,F}$ ($\rho \in {\mathcal B}_{I',I,F}$) and $\rho \subseteq \rho'$ imply $\rho' \in {\mathcal A}_{I',I,F}$ ($\rho' \in {\mathcal B}_{I',I,F}$). The set  ${\mathcal A}_{I',I,F}$ (${\mathcal B}_{I',I,F}$), equipped with the relation of set inclusion, becomes a partially ordered set; for the sake of brevity, we will not consider the minimal elements and the least elements (if exist) of these partially ordered sets here. The interested reader may try to construct some examples concerning this issue.

Now we will state and prove some parts of the following theorem for Bohr $({\mathcal B},I',\rho)$-almost periodic functions ($({\mathcal B},I',\sigma)$-uniformly recurrent functions):

\begin{thm}\label{lojalni}
Suppose that $\emptyset  \neq I' \subseteq {\mathbb R}^{n},$ $\emptyset  \neq I \subseteq {\mathbb R}^{n},$ $F : I \times X \rightarrow Y$ is Bohr $({\mathcal B},I',\rho)$-almost periodic ($({\mathcal B},I',\rho)$-uniformly recurrent), $\rho$ is a binary relation on $Y$ and $I +I' \subseteq I.$ Then the following holds:
\begin{itemize}
\item[(i)] Set $\sigma:=\{ (\| y_{1}\|_{Y}, \| y_{2}\|_{Y}) \ | \  \exists {\bf t}\in I \ \exists x\in X \ : \  y_{1}=F({\bf t};x)\mbox{ and }y_{2}\in \rho(y_{1}) \}.$ Then the function 
$\|F(\cdot;\cdot)\|_{Y}$ is Bohr $({\mathcal B},I',\sigma)$-almost periodic ($({\mathcal B},I',\sigma)$-uniformly recurrent).
\item[(ii)] Suppose that $\lambda \in {\mathbb C} \setminus \{0\}.$ Set
$\rho_{\lambda}:=\{ \lambda ( y_{1}, y_{2}) \ | \  \exists {\bf t}\in I \ \exists x\in X \ : \  y_{1}=F({\bf t};x)\mbox{ and }y_{2}\in \rho(y_{1}) \}.$
Then the function $\lambda F(\cdot ;\cdot)$ is Bohr $({\mathcal B},I',\rho_{\lambda})$-almost periodic ($({\mathcal B},I',\rho_{\lambda})$-uniformly recurrent).
\item[(iii)] Suppose $a\in {\mathbb C}$ and $x_{0}\in X.$ Define $G :(I-a) \times X \rightarrow Y$ by $G({\bf t};x):=F({\bf t}+a;x+x_{0}),$ ${\bf t}\in I-a,$ $x\in X,$ as well as
${\mathcal B}_{x_{0}}:=\{-x_{0}+B : B\in {\mathcal B}\},$ $I_{a}':=I'$ and
$\rho_{a,x_{0}}:=\{ ( y_{1}, y_{2}) \ | \  \exists {\bf t}\in I-a \ \exists x\in X \ : \  y_{1}=F({\bf t}+a;x+x_{0})\mbox{ and }y_{2}\in \rho(y_{1}) \}.$
Then the function $G(\cdot ;\cdot)$ is Bohr $({\mathcal B}_{x_{0}},I'_{a},\rho_{a,x_{0}})$-almost periodic ($({\mathcal B}_{x_{0}},I'_{a},\rho_{a,x_{0}})$-uniformly recurrent).
\item[(iv)] Suppose that $a,\ b\in {\mathbb C}\setminus \{0\}.$ Define the function $G :(I/a) \times X \rightarrow Y$ by $G({\bf t};x):=F(a{\bf t};bx),$ ${\bf t}\in I/a,$ $x\in X,$ as well as
${\mathcal B}_{b}:=\{b^{-1}B : B\in {\mathcal B}\},$ $I_{a}':=I'/a$ and
$\rho_{a,b}:=\{ ( y_{1}, y_{2}) \ | \  \exists {\bf t}\in I/a \ \exists x\in X \ : \  y_{1}=F(a{\bf t};bx)\mbox{ and }y_{2}\in \rho(y_{1}) \}.$
Then the function $G(\cdot ;\cdot)$ is Bohr $({\mathcal B}_{b},I'_{a},\rho_{a,b})$-almost periodic ($({\mathcal B}_{b},I'_{a},\rho_{a,b})$-uniformly recurrent).
\item[(v)] Assume that for each $B\in {\mathcal B}$ there exists $\epsilon_{B}>0$ such that
the sequence $(F_{k}(\cdot ;\cdot))$ of Bohr $({\mathcal B},I',\rho)$-almost periodic functions ($({\mathcal B},I',\rho)$-uniformly recurrent functions)
converges uniformly to a function $F(\cdot ;\cdot)$ on the set $B^{\circ} \cup \bigcup_{x\in \partial B}B(x,\epsilon_{B}).$ Then the function $F(\cdot;\cdot)$ is 
Bohr $({\mathcal B},I',\rho)$-almost periodic ($({\mathcal B},I',\rho)$-uniformly recurrent), provided that $D(\rho)$ is a closed subset of $Y$ and $\rho$ is continuous in the following sense: 
\begin{itemize}
\item[($C_{\rho}$)]
For each $\epsilon>0$ there exists $\delta>0$ such that, for every $y_{1},\ y_{2}\in Y$ with $\| y_{1}-y_{2}\|_{Y}<\delta$, we have $\| z_{1}-z_{2}\|_{Y}<\epsilon/3$ for every $z_{1}\in \rho(y_{1})$ and $z_{2}\in \rho(y_{2}).$ 
\end{itemize}
\end{itemize}
\end{thm}

\begin{proof}
The proofs of parts (i)-(iv) are almost trivial and therefore omitted. Now we will prove (v).  Due to the proofs of \cite[Proposition 2.7, Proposition 2.8]{marko-manuel-ap}, it follows that the function $F(\cdot;\cdot)$ is continuous. Let $\epsilon>0$ and $B\in {\mathcal B}$ be fixed. Since 
$D(\rho)$ is a closed subset of $Y$ and
the sequence $(F_{k}(\cdot ;\cdot))$ 
converges uniformly to a function $F(\cdot ;\cdot)$ on the set $B^{\circ} \cup \bigcup_{x\in \partial B}B(x,\epsilon_{B}),$ we have that $F({\bf t};x)\in D(\rho)$ for all ${\bf t}\in I$ and $x\in X.$ Further on, let a number $\delta>0$ be chosen in accordance with the continuity of relation $\rho.$ 
For $\epsilon_{0}\equiv \min(\epsilon/3,\delta),$ we can find a positive integer $k\in {\mathbb N}$ such that $\| F_{k}({\bf t};x)-F({\bf t};x)\|_{Y}<\epsilon_{0}.$ Let ${\bf t}\in I$ and $x\in B$ be fixed. Let
$y^{k}_{{\bf t};x} \in \rho(F_{k}({\bf t};x))$ be the element determined from the definition of Bohr $({\mathcal B},I',\rho)$-almost periodicity of function $F_{k}(\cdot;\cdot),$ with the number $\epsilon$ replaced with the number $\epsilon/3$ therein.
Pick up now an arbitrary element $y_{{\bf t};x}$ from $\rho(F({\bf t};x)).$ 
Then we have (a point $\tau \in I'$ satisfies the requirements in Definition \ref{nafaks123456789012345} for the function $F_{k}(\cdot;\cdot)$):
\begin{align*}
\bigl\| F({\bf t}+\tau; x) -y_{{\bf t};x}\bigr\|_{Y} & \leq 
\bigl \| F({\bf t}+\tau; x) -F_{k}({\bf t}+\tau; x) \bigr\|_{Y}+\bigl \| F_{k}({\bf t}+\tau; x) -  y^{k}_{{\bf t};x}\bigr\|_{Y}\\ & +\bigl\| y_{{\bf t};x}-y^{k}_{{\bf t};x}\bigr\|_{Y}
\leq \epsilon_{0}+(\epsilon/3)+(\epsilon/3)\leq \epsilon. 
\end{align*}
This completes the proof of (v). 
\end{proof}

The proof of following proposition is simple and therefore omitted:

\begin{prop}\label{superstebagT}
Suppose that $\emptyset  \neq I' \subseteq {\mathbb R}^{n},$ $\emptyset  \neq I \subseteq {\mathbb R}^{n},$ $F : I \times X \rightarrow Y$ is a continuous function, $\rho$ is a binary relation on $Y$ and $I +I' \subseteq I.$
If $F : I \times X \rightarrow Y$ is Bohr $({\mathcal B},I',\rho)$-almost periodic/$({\mathcal B},I',\rho)$-uniformly recurrent, and $\phi : Y \rightarrow Z$ is uniformly continuous on the set $R(F) \cup \rho(R(F)),$ then 
$\phi \circ F : I \times X \rightarrow Z$ is Bohr $({\mathcal B},I', \phi \circ \rho)$-almost periodic/$({\mathcal B},I', \phi \circ \rho)$-uniformly recurrent.
\end{prop}

The supremum formula for $({\mathcal B},I',\rho)$-uniformly recurrent functions reads as follows:

\begin{prop}\label{deb}
Suppose that $\emptyset  \neq I' \subseteq {\mathbb R}^{n},$ $\emptyset  \neq I \subseteq {\mathbb R}^{n},$ $I +I' \subseteq I$ and $F : I \times X \rightarrow Y$ is a $({\mathcal B},I',\rho)$-uniformly recurrent function, where $\rho =T\in L(Y)$ is a linear isomorphism. Then for each real number $a>0$ we have:
\begin{align}\label{tupak12345ceT}
\sup_{{\bf t}\in I,x\in B}\bigl\|F({\bf t};x) \bigr\|_{Y}\leq \sup_{{\bf t}\in I+I',|{\bf t}|\geq a ,x\in B}\bigl\| T^{-1}F({\bf t};x) \bigr\|_{Y},
\end{align}
and for each $x\in X$ we have
\begin{align}\label{tupak12345ceT}
\sup_{{\bf t}\in I}\bigl\|F({\bf t};x) \bigr\|_{Y}\leq \sup_{{\bf t}\in I+I',|{\bf t}|\geq a}\bigl\| T^{-1}F({\bf t};x) \bigr\|_{Y},
\end{align}
so that
the function $F(\cdot;x)$ is identically equal to zero provided that the function $F(\cdot;\cdot)$ is $({\mathcal B},I',\rho)$-uniformly recurrent and $\lim_{|{\bf t}|\rightarrow +\infty, {\bf t}\in I+I'}F({\bf t};x)=0.$
\end{prop}

\begin{proof}
Let $a>0,$ $\epsilon>0$ and $B\in {\mathcal B}$ be given. Then  
there exists a sequence $({\bf \tau}_{k})$ in $I'$ such that $\lim_{k\rightarrow +\infty} |{\bf \tau}_{k}|=+\infty$ and that, for every ${\bf t}\in I$ and $x\in B,$ we have that \eqref{oblak1} holds with $y_{{\bf t};x}=TF({\bf t};x).$ This implies the existence of an integer $k\in {\mathbb N}$ such that
\begin{align}\label{oblak11}
\bigl\|F({\bf t}+{\bf \tau}_{k};x)-TF({\bf t};x)\bigr\|_{Y} \leq \epsilon ,\quad { \bf t}\in I,\ x\in B.
\end{align}
The operator $T$ is a linear isomorphism, so that \eqref{oblak11} immediately implies
\begin{align*}
 \bigl\|T^{-1}F({\bf t}+{\bf \tau}_{k};x)-F({\bf t};x)\bigr\|_{Y} \leq \epsilon \cdot \bigl\| T^{-1}\bigr \| ,\quad { \bf t}\in I,\ x\in B
\end{align*}
and
\begin{align*}
\bigl\|F({\bf t};x) \bigr\|_{Y}\leq \bigl\| T^{-1}F({\bf t}+{\bf \tau}_{k};x) \bigr\|_{Y}+\epsilon \cdot \bigl\| T^{-1}\bigr \|, \quad { \bf t}\in I,\ x\in B.
\end{align*}
Since $\epsilon>0$ was arbitrary, this yields
\begin{align*}
\bigl\|F({\bf t};x) \bigr\|_{Y}\leq \bigl\| T^{-1}F({\bf t}+{\bf \tau}_{k};x) \bigr\|_{Y} ,\quad { \bf t}\in I,\ x\in B
\end{align*}
and
\eqref{tupak12345ceT}. The remainder of proof, for a fixed element $x\in X,$ follows from the same arguments and the existence of a set $B\in {\mathcal B}$ such that $x\in B.$
\end{proof}

Regarding the convolution invariance of Bohr $({\mathcal B},I',\rho)$-almost periodic ($({\mathcal B},I',\rho)$-uniformly recurrent) functions, we will clarify the following result:

\begin{thm}\label{milenko}
Suppose that $h\in L^{1}({\mathbb R}^{n})$ and $F : {\mathbb R}^{n} \times X \rightarrow Y$ is a continuous function satisfying that for each $B\in {\mathcal B}$ there exists a finite real number $\epsilon_{B}>0$ such that
$\sup_{{\bf t}\in {\mathbb R}^{n},x\in B^{\cdot}}\|F({\bf t},x)\|_{Y}<+\infty,$
where $B^{\cdot} \equiv B^{\circ} \cup \bigcup_{x\in \partial B}B(x,\epsilon_{B}).$ Suppose, further, that $\rho=A$ is a closed linear operator on $Y$ satisfying that:
\begin{itemize}
\item[(B)] For each ${\bf t}\in {\mathbb R}^{n}$ and $x\in B,$ the function ${\bf s} \mapsto AF({\bf t}
-{\bf s};x),$ ${\bf s}\in {\mathbb R}^{n}$ is Bochner integrable.
\end{itemize}
Then the function
\begin{align}\label{gariprekrsaj}
(h\ast F)({\bf t};x):=\int_{{\mathbb R}^{n}}h(\sigma) F({\bf t}-\sigma;x)\, d\sigma,\quad {\bf t}\in {\mathbb R}^{n},\ x\in X 
\end{align}
is well defined and for each $B\in {\mathcal B}$ we have $\sup_{{\bf t}\in {\mathbb R}^{n},x\in B^{\cdot}}\|(h\ast F)({\bf t};x)\|_{Y}<+\infty;$ furthermore, if $F(\cdot;\cdot)$ is 
Bohr $({\mathcal B},I',A)$-almost periodic ($({\mathcal B},I',A)$-uniformly recurrent), then 
the function $(h\ast F)(\cdot;\cdot)$ is Bohr $({\mathcal B},I',A)$-almost periodic ($({\mathcal B},I',A)$-uniformly recurrent).
\end{thm}

\begin{proof}
We will consider only Bohr $({\mathcal B},I',A)$-almost periodic functions. The function $(h\ast F)(\cdot;\cdot)$ is well defined and $\sup_{{\bf t}\in {\mathbb R}^{n},x\in B^{\cdot}}\|(h\ast F)({\bf t};x)\|_{Y}<+\infty$ for all $B\in {\mathcal B}$. The continuity of this function at the fixed point $({\bf t}_{0};x_{0}) \in {\mathbb R}^{n} \times X$ follows from the existence of a set $B\in {\mathcal B}$ such that $x_{0}\in B,$ the assumption $\sup_{{\bf t}\in {\mathbb R}^{n},x\in B^{\cdot}}\|F({\bf t};x)\|_{Y}<+\infty$  and the dominated convergence theorem. Let $\epsilon>0$ and $B\in {\mathcal B}$ be given. Then there exists 
$l>0$ such that for each ${\bf t}_{0} \in I'$ there exists ${\bf \tau} \in B({\bf t}_{0},l) \cap I'$ such that, for every ${\bf t}\in {\mathbb R}^{n}$ and $x\in B,$ there exists an element $y_{{\bf t};x}=A(F({\bf t};x))$ such that
\eqref{oblak}
holds. Due to Lemma \ref{stana} and condition (C),
for every ${\bf t}\in {\mathbb R}^{n}$ and $x\in B,$ we have that 
$z_{{\bf t},x}:=A((h\ast F)({\bf t};x))=\int_{ {\mathbb R}^{n}}h({\bf s})A(F({\bf t}-{\bf s};x))\, d{\bf s}.$
Therefore, we have
\begin{align*}
\Bigl\| & (h\ast F)({\bf t}+\tau ; x)-z_{{\bf t},x}\Bigr\|_{Y}
\\& \leq \int_{{\mathbb R}^{n}}|h(\sigma)| \cdot \bigl\| F({\bf t}+\tau -{\bf s}; x)-A(F({\bf t}-{\bf s};x))\bigr\|_{Y} \, d{\bf s}
\\& \leq \epsilon \cdot \| h\|_{L^{1}({\mathbb R}^{n})},\quad {\bf t}\in {\mathbb R}^{n},\ x\in B,
\end{align*}
which completes the proof.
\end{proof}

\begin{rem}\label{opaska}
The requirements of Theorem \ref{milenko} are satisfied if $A\in L(X).$
\end{rem}

For the sake of completeness, we will provide all relevant details of the following result, which has numerous important applications in the analysis of the existence and uniqueness of time $A$-almost periodic solutions for various classes of abstract (degenerate) Volterra integro-differential equations; the statement can be extended for the corresponding Stepanov classes which will be considered somewhere else:

\begin{thm}\label{mkmk}
Let $\emptyset \neq I'\subseteq {\mathbb R}^{n},$ let $A$ be a closed linear operator on $X,$ and let $(R({\bf t}))_{{\bf t}> {\bf 0}}\subseteq L(X)$ be a strongly continuous operator family such that $R({\bf t})A\subseteq AR({\bf t})$ for all ${\bf t}\in {\mathbb R}^{n}$  and
$\int_{(0,\infty)^{n}}\|R({\bf t} )\|\, d{\bf t}<\infty .$ 
If $f : {\mathbb R}^{n} \rightarrow X$ is a bounded $(I',A)$-almost periodic function, resp.  bounded $(I',A)$-uniformly recurrent function, and the function $Af : {\mathbb R}^{n} \rightarrow X$ is well defined and bounded, then the function $F: {\mathbb R}^{n} \rightarrow X,$ given by
\eqref{pikford},
is well-defined, bounded and $(I',A)$-almost periodic, resp. well-defined, bounded and $(I',A)$-uniformly recurrent.
\end{thm}

\begin{proof}
We will consider only $(I',A)$-almost periodic functions.
It is clear that the function $F(\cdot)$ is well-defined and bounded since
$$
F({\bf t}):=\int_{(0,\infty)^{n}}R({\bf s})f({\bf t}-{\bf s})\, d{\bf s},\quad {\bf t}\in {\mathbb R}^{n},
$$
as well as $\int_{(0,\infty)^{n}}\|R({\bf t} )\|\, d{\bf t}<\infty $ and the function $f(\cdot)$ is bounded. Similarly, Lemma \ref{stana}, our assumptions $R({\bf t})A\subseteq AR({\bf t})$ for all ${\bf t}\in {\mathbb R}^{n}$ and the boundedness of function $Af(\cdot)$ together imply that $AF({\bf t})=\int_{(0,\infty)^{n}}R({\bf s})Af({\bf t}-{\bf s})\, d{\bf s}$ for all $ {\bf t}\in {\mathbb R}^{n}.$ 
Furthermore, for every $\tau \in I'$ and ${\bf t}\in {\mathbb R}^{n},$ we have
\begin{align*}
\| F({\bf t}+\tau)-AF({\bf t})\| &=\Biggl\| \int_{(0,\infty)^{n}}R({\bf s})\bigl[f({\bf t}+\tau-{\bf s})-Af({\bf t}-{\bf s})\bigr]\, d{\bf s}\Biggr\|
\\ & \leq \int_{(0,\infty)^{n}}\| R({\bf s}) \| \cdot \| f({\bf t}+\tau-{\bf s})-Af({\bf t}-{\bf s}) \|\, d{\bf s}.
\end{align*}
Keeping in mind the corresponding definition of $(I',A)$-almost periodicity, the above calculation simply completes the proof of theorem.
\end{proof}

Suppose that $F : I \times X \rightarrow Y$ and $G : I \times Y \rightarrow Z$ are given continuous functions. Then we define the multi-dimensional Nemytskii operator
$W : I  \times X \rightarrow Z$ by
\begin{align}\label{skadar}
W({\bf t}; x):=G\bigl({\bf t} ; F({\bf t}; x)\bigr),\quad {\bf t} \in I,\ x\in X.
\end{align}

The following composition principle slightly generalizes the statement of \cite[Theorem 2.19]{nova-mse}; keeping in mind the corresponding definitions, the proof is almost the same as the proof of this theorem and therefore omitted:

\begin{thm}\label{episkop-jovan}
Suppose that the functions $F : I \times X \rightarrow Y$ and $G : I \times Y \rightarrow Z$ are continuous,  $\emptyset \neq I' \subseteq {\mathbb R}^{n},$ $\emptyset '\subseteq I \subseteq {\mathbb R}^{n}$ and $\sigma$ is a binary relation on $Z.$
\begin{itemize}
\item[(i)]
Suppose further that, for every $B\in {\mathcal B}$ and $\epsilon>0,$
there exists $l>0$ such that for each ${\bf t}_{0} \in I'$ there exists ${\bf \tau} \in B({\bf t}_{0},l) \cap I'$ such that, 
for every ${\bf t}\in I$ and $x\in B,$ there exist elements $y_{{\bf t};x}\in \rho (F({\bf t};x))$ and $z_{{\bf t};x}\in \sigma (W({\bf t};x))$ such that
\eqref{oblak} holds
as well as that
there exists a finite real constant $L>0$ such that
\begin{align}\label{lajbaha}
\bigl\|G({\bf t}+\tau;F({\bf t}+\tau ; x))-G\bigl({\bf t}+\tau;y_{{\bf t},x}\bigr)\bigr\|_{Z} \leq L\bigl\|F({\bf t}+{\bf \tau};x)-y_{{\bf t};x}\bigr\|_{Y}
\end{align}
and 
\begin{align}\label{ujshe1}
\bigl\|G({\bf t}+{\bf \tau};y_{{\bf t},x})-z_{{\bf t};x}\bigr\|_{Z} \leq \epsilon .
\end{align}
Then the function $W(\cdot;\cdot),$ given by \eqref{skadar}, is Bohr $({\mathcal B},I',\sigma)$-almost periodic.
\item[(ii)] 
Suppose that, for every $B\in {\mathcal B},$ 
there exists a sequence $({\bf \tau}_{k})$ in $I'$ such that\\ $\lim_{k\rightarrow +\infty} |{\bf \tau}_{k}|=+\infty$ and that, for every ${\bf t}\in I$ and $x\in B,$ there exist elements $y_{{\bf t};x}\in \rho (F({\bf t};x))$ and $z_{{\bf t};x}\in \sigma (W({\bf t};x))$ such that
\eqref{oblak1} holds as well as that for each $k\in {\mathbb N}$ the equations
\eqref{lajbaha}-\eqref{ujshe1} hold with the number $\tau$ replaced with the number $\tau_{k}$ therein. 
Then the function $W(\cdot;\cdot),$ given by \eqref{skadar}, is $({\mathcal B},I',\sigma)$-uniformly recurrent.
\end{itemize}
\end{thm} 

\subsection{$T$-Almost periodic type functions in finite-dimensional spaces}\label{nova}

In this subsection, we will clarify some basic results concerning $T$-almost periodic type functions of form $F: I \rightarrow {\mathbb C}^{k}$, where $k\in {\mathbb N}$ and $I=[0,\infty)$ or $I={\mathbb R}$. We will also provide several illustrative examples in this direction.

First of all,
note that the argumentation contained in the proof of \cite[Proposition 2.6]{c1} enables one to deduce the following result (let us only point out that, if $T$ is a linear isomorphism, then the estimate $\|f(t)\|\leq \| T^{-1}\| \cdot \| Tf(t)\|,$ $t\in I$ can be used):

\begin{prop}\label{tomie}
Suppose that 
$\rho=T\in L(Y)$,
$I=[0,\infty)$ or $I={\mathbb R},$ and $I'=[0,\infty).$ If the function $F : I \rightarrow Y$ is $(I',T)$-uniformly recurrent and $F\neq 0,$ then $\|  T\| _{L(Y)}\geq 1;$ furthermore, if $T$ is a linear isomorphism, then $\| T^{-1}\|_{L(Y)}\geq 1.$ 
\end{prop} 

For the linear continuous operators which are not scalar multiples of the identity operator, case in which  $\|  T\| _{L(Y)}> 1$ is possible:

\begin{example}\label{matrice}
Suppose that $Y:={\mathbb C}^{2}$ is equipped with the norm $\| (z_{1},z_{2})\|_{Y}:=\sqrt{|z_{1}|^{2}+|z_{2}|^{2}}$ ($z_{1},\ z_{2}\in {\mathbb C}$), $a\in {\mathbb C}$ satisfies $|a|>1,$ the function $u : [0,\infty) \rightarrow {\mathbb C}$ is almost periodic, $F(t):=(u(t),u(t)),$ $t\geq 0,$ and
\begin{align}\label{ispost}
T:=\Biggl[\begin{matrix}
a\  & \ 1-a\\ a &\  1-a \end{matrix}\Biggr].
\end{align}
Then the function $F(\cdot)$ is $T$-almost periodic but $\| T\|_{L(Y)}>1.$
Furthermore, the assumption that $T$ is a linear isomorphism does not imply $\| T^{-1}\|_{L(Y)}= 1;$ consider the same pivot space $Y$, the same function $F(\cdot)$ and the matrix
 $$
T:=\Biggl[\begin{matrix}
2\  & \ -1\\ 1 &\  0 \end{matrix}\Biggr].$$ Then we have both
$\| T\|_{L(Y)}>1$ and $\| T^{-1}\|_{L(Y)}>1.$
\end{example}

It is clear that Proposition \ref{tomie} provides a large class of complex matrices $T=A$ of format $k\times k$ ($k\in {\mathbb N}$) such that the only $(I',A)$-uniformly recurrent function $F : [0,\infty) \rightarrow {\mathbb C}^{k}$ is the zero function, actually. Now we will state and prove the following result:

\begin{prop}\label{tomqa}
Suppose that $k\in {\mathbb N},$ $T=A=[a_{ij}]$ is a non-zero complex matrix of format $k\times k,$ $I={\mathbb R}$ or $I=[0,\infty),$ $I'\subseteq {\mathbb R},$ $I+I'\subseteq I,$ and the function $F: I \rightarrow {\mathbb C}^{k}$ is Bohr $(I',A)$-almost periodic (bounded $(I',A)$-uniformly recurrent). If $F=(F_{1},\cdot \cdot \cdot, F_{k})$, then there exists a non-trivial linear combination of functions $F_{1},\cdot \cdot \cdot,F_{k}$ which is Bohr $(I',{\rm I})$-almost periodic (bounded $(I',{\rm I})$-uniformly recurrent).
\end{prop}

\begin{proof}
We will prove the statement only for $(I',A)$-almost periodicity. Let $\epsilon>0$ be given. Then there exists a finite real number $l>0$ such that, for every $t_{0}\in I,$ there exists a point $\tau \in B(t_{0},l) \cap I'$ such that 
\begin{align}\label{doros}
\bigl| F_{i}(t+\tau)-[a_{i1}F_{1}(t)+\cdot \cdot \cdot +a_{ik}F_{k}(t)] \bigr|\leq \epsilon\mbox{ for all }t\in I \mbox{ and }i\in {\mathbb N}_{k}.
\end{align}
Suppose that $\lambda \in \sigma_{p}(A) \setminus \{0\},$ where $\sigma_{p}(A)$ denotes the point spectrum of $A;$ such a number $\lambda$ exists since $A\neq 0.$ Then there exists a tuple $(\alpha_{1},\cdot \cdot \cdot,\alpha_{k}) \in {\mathbb R}^{k} \setminus \{(0,\cdot \cdot \cdot,0)\}$ such that $\alpha_{1}a_{i1}+\alpha_{2}a_{i2}+\cdot \cdot \cdot +\alpha_{i}(a_{ii}-\lambda)+\cdot \cdot \cdot+\alpha_{k} a_{ik}=0$ for all $i\in {\mathbb N}_{k};$ $(\alpha_{1},\cdot \cdot \cdot,\alpha_{k})$ is, in fact, an eigenvector of matrix $A$ which corresponds to the eigenvalue $\lambda.$ Multiplying \eqref{doros} with $\alpha_{i},$
and adding all obtained inequalities for $i=1,\cdot \cdot \cdot, k,$ we get that the function $u(t):=\alpha_{1}F_{1}(t)+\cdot \cdot \cdot +\alpha_{k}F_{k}(t),$ $t\in I$ satisfies $|u(t+\tau)-\lambda u(t)|\leq \epsilon (|\alpha_{1}|+\cdot \cdot \cdot +|\alpha_{k}|)$ for all $t\in I.$ If $|\lambda| \neq 1,$ then \cite[Proposition 2.6]{c1} immediately gives $u(t)\equiv 0$ and the proof is completed. If $|\lambda|=1,$ then the result simply follows by applying \cite[Corollary 2.10, Proposition 2.11]{c1}.
\end{proof}

Using the conclusion given directly after Corollary \ref{rtanj}, we have that the terms ``bounded $(I',A)$-uniformly recurrent'' and ``bounded $(I',{\rm I})$-uniformly recurrent'' can be replaced with the terms ``$(I',A)$-uniformly recurrent'' and ``$(I',{\rm I})$-uniformly recurrent'' in the case that $I={\mathbb R},$ when the function $F(\cdot)$ is $(I'-I',{\rm I})$-almost periodic ((bounded) $(I'-I',{\rm I})$-uniformly recurrent). If $I=[0,\infty),$ $A$ is invertible and $F(\cdot)$ is uniformly continuous, then the function
$F(\cdot)$ is $((I'-I') \cap [0,\infty),{\rm I})$-almost periodic ((bounded) $((I'-I') \cap [0,\infty),{\rm I})$-uniformly recurrent),
which follows from an application of Corollary \ref{rtanj}, and Theorem \ref{lenny-jassonceT} below.

If the matrix $A$ is not invertible and $I=[0,\infty)$, then the case in which any of the functions $F_{1}(\cdot),\cdot \cdot \cdot, F_{k}(\cdot)$ is not $(I',{\rm I})$-almost periodic (bounded $(I',{\rm I})$-uniformly recurrent) can occur, which can be easily shown by using Proposition \ref{adding} (see also Remark \ref{tomka}), so that there exists an $A$-almost periodic function (bounded $A$-uniformly recurrent function) $F : [0,\infty) \rightarrow {\mathbb C}^{k},$ which is not almost periodic (uniformly recurrent):

\begin{example}\label{tomqqa}
Let $Y:={\mathbb C}^{2}$, let $a$, $u : [0,\infty) \rightarrow {\mathbb C}$ and $F(\cdot):=(u(\cdot),u(\cdot))$ possess the same meaning as in Example \ref{matrice}. Further on, let $I=[0,\infty),$ $I'=(0,\infty),$ and let the matrix $T=A$ be given by \eqref{ispost}.
Then $N(A)=\{ (\alpha,\beta) \in {\mathbb C}^{2} : \alpha a+ \beta (1-a)=0 \}.$ Suppose that $q=(q_{1},q_{2}) : [0,\infty) \rightarrow N(A)$
is any continuous function tending to zero as the norm of the argument goes to plus infinity. Then Proposition \ref{adding} implies that the function $t\mapsto (u(t)+q_{1}(t),u(t)+q_{2}(t)),$ $t\geq 0$ is also 
$(I',A)$-almost periodic. It is clear that this function cannot be almost periodic in the case that some of the functions $q_{1}(\cdot)$ or $q_{2}(\cdot)$ is not identically equal to the zero function.
\end{example}

\subsection{${\mathbb D}$-Asymptotically Bohr $({\mathcal B},I',\rho)$-almost periodic type functions}\label{mono-mono}

We start this subsection by introducing the following notion: 

\begin{defn}\label{kakavsam jadebil1}
Suppose that ${\mathbb D} \subseteq I \subseteq {\mathbb R}^{n},$ the set ${\mathbb D}$  is unbounded,
$\emptyset  \neq I' \subseteq {\mathbb R}^{n},$ $\emptyset  \neq I \subseteq {\mathbb R}^{n},$ $F : I \times X \rightarrow Y$ is a continuous function, $\rho$ is a binary relation on $Y$ and $I +I' \subseteq I.$ Then we say that
the function $F(\cdot)$ is (strongly) ${\mathbb D}$-asymptotically Bohr $({\mathcal B},I',\rho)$-almost periodic, resp. 
(strongly) ${\mathbb D}$-asymptotically $({\mathcal B},I',\rho)$-uniformly recurrent,
if and only if there exists
a 
Bohr $({\mathcal B},I',\rho)$-almost periodic function, resp. $({\mathcal B},I',\rho)$-uniformly recurrent function, ($F_{0} : {\mathbb R}^{n} \times X \rightarrow Y$)
$F_{0} : I \times X \rightarrow Y$ and a function $Q\in C_{0,{\mathbb D},{\mathcal B}}(I \times X : Y)$ such that 
$F({\bf t};x)=F_{0}({\bf t};x)+Q({\bf t};x),$ ${\bf t} \in I,$ $x\in X.$

The functions $F_{0}(\cdot;\cdot)$ and $Q(\cdot ; \cdot)$ are usually called the principal part of $F(\cdot;\cdot)$ and the corrective (ergodic) part of $F(\cdot;\cdot),$ respectively.
\end{defn}

We will not reconsider here \cite[Theorem 2.22]{nova-mse} for ${\mathbb D}$-asymptotically Bohr\\ $({\mathcal B},I',\rho)$-almost periodic functions and
${\mathbb D}$-asymptotically $({\mathcal B},I',\rho)$-uniformly recurrent functions. 

In the following result, which is applicable to the general binary relations satisfying conditions clarified in Theorem \ref{lojalni}(v), we follow a new approach based on the use of supremum formula (the argumentation contained in the proof of \cite[Theorem 4.29]{diagana} can be used only for the binary relations $\rho=T\in L(Y)$ which are linear isomorphisms; see e.g., \cite[Proposition 2.27(ii)]{marko-manuel-ap} and \cite[Proposition 2.24(ii)]{nova-mse}, where we have also assumed that $I'=I$):

\begin{thm}\label{kozjak}
Suppose that $\rho$ is a binary relation on $Y$ satisfying that $D(\rho)$ is a closed subset of $Y$ and 
condition
($C_{\rho}$). Suppose, further, that
for each integer $j\in {\mathbb N}$ the function $F_{j}(\cdot ; \cdot)=G_{j}(\cdot ; \cdot)+Q_{j}(\cdot ; \cdot)$ is $I$-asymptotically  Bohr $({\mathcal B},I',\rho)$-almost periodic ($I$-asymptotically $({\mathcal B},I',\rho)$-uniformly recurrent), where $G_{j}(\cdot ; \cdot)$ is Bohr $({\mathcal B},I',\rho)$-almost periodic ($({\mathcal B},I',\rho)$-uniformly recurrent) and $Q_{j}\in C_{0,I,{\mathcal B}}(I\times X : Y).$
Let for each $B\in {\mathcal B}$ there exist $\epsilon_{B}>0$ such that
the sequence $(F_{j}(\cdot ;\cdot))$ converges uniformly to a function $F(\cdot ;\cdot)$ on the set $B^{\circ} \cup \bigcup_{x\in \partial B}B(x,\epsilon_{B}),$ and let $Q_{j}\in C_{0,I,{\mathcal B}^{\circ}}(I\times X : Y),$
where ${\mathcal B}^{\circ}\equiv \{ B^{\circ} : B\in {\mathcal B}\}.$
If
for each natural numbers $m,\ k\in {\mathbb N}$ the function $G_{k}(\cdot ; \cdot)-G_{m}(\cdot ; \cdot)$ satisfies the following supremum formula:
\begin{itemize}
\item[(S)] for every $a>0,$ we have 
$$
\sup_{{\bf t}\in I,x\in B^{\circ}}\bigl\| G_{k}({\bf t} ; x)-G_{m}({\bf t} ; x)\bigr \|_{Y}=\sup_{{\bf t}\in I,|{\bf t}|\geq a,\ x\in B^{\circ}}\bigl\| G_{k}({\bf t} ; x)-G_{m}({\bf t} ; x)\bigr \|_{Y}
$$
\end{itemize}
then $F(\cdot ;\cdot)$ is $I$-asymptotically Bohr $({\mathcal B}^{\circ},I',\rho)$-almost periodic ($I$-asymptotically\\ $({\mathcal B}^{\circ},I',\rho)$-uniformly recurrent).
\end{thm}

\begin{proof}
Let $\epsilon>0$ and $B\in {\mathcal B}$  be given. Then there exists a natural number $k_{0}\in {\mathbb N}$ such that, for every natural numbers $k,\ m\in {\mathbb N}$ with $\min(k,m)\geq k_{0},$ we have
\begin{align}\label{zeljo}
\sup_{{\bf t}\in I,\ x\in B^{\circ}}\bigl\| F_{k}({\bf t} ; x)-F_{m}({\bf t} ; x)\bigr \|_{Y}<\epsilon/3.
\end{align}
Let $k,\ m\in {\mathbb N}$ with $\min(k,m)\geq k_{0}$ be fixed. Then there exists a finite real number $a_{k,m}>0$ such that 
\begin{align}\label{chist}
\sup_{{\bf t}\in I,|{\bf t}|\geq a_{k,m},\ x\in B^{\circ}}\bigl\| Q_{k}({\bf t} ; x)\bigr\|_{Y}\leq \epsilon/3 \mbox{  and  }\sup_{{\bf t}\in I,|{\bf t}|\geq a_{k,m},\ x\in B^{\circ}}\bigl\| Q_{m}({\bf t} ; x)\bigr\|_{Y}\leq \epsilon/3.
\end{align}
Keeping in mind \eqref{zeljo}-\eqref{chist}, we get: 
\begin{align*}
&\sup_{{\bf t}\in I,|{\bf t}|\geq a_{k,m},\ x\in B^{\circ}}\bigl\| G_{k}({\bf t} ; x) -G_{m}({\bf t} ; x)\bigr \|_{Y}
\\& \leq (\epsilon/3)+\sup_{{\bf t}\in I,|{\bf t}|\geq a_{k,m},\ x\in B^{\circ}}\bigl\| Q_{k}({\bf t} ; x)-Q_{m}({\bf t} ; x)\bigr \|_{Y}
\\ & \leq (\epsilon/3)+\sup_{{\bf t}\in I,|{\bf t}|\geq a_{k,m},\ x\in B^{\circ}}\bigl\| Q_{k}({\bf t} ; x)\bigr\|_{Y}+\sup_{{\bf t}\in I,|{\bf t}|\geq a_{k,m},\ x\in B^{\circ}}\bigl\| Q_{m}({\bf t} ; x)\bigr\|_{Y}\leq \epsilon.
\end{align*}
Using this estimate and the supremum formula (S), we get that 
\begin{align}\label{zeljoT}
\sup_{{\bf t}\in I,\ x\in B^{\circ}}\bigl\| G_{k}({\bf t} ; x)& -G_{m}({\bf t} ; x)\bigr \|_{Y}\leq \epsilon.
\end{align}
Therefore, the sequence $(G_{k}({\bf t};x))$ is Cauchy and therefore convergent for each ${\bf t}\in I$ and $x\in X.$ If we denote by $G(\cdot; \cdot)$ the corresponding limit function, then \eqref{zeljoT} yields
\begin{align*}
\sup_{{\bf t}\in I,\ x\in B^{\circ}}\bigl\| G_{k}({\bf t} ; x)& -G({\bf t} ; x)\bigr \|_{Y}\leq \epsilon.
\end{align*}
Applying Theorem \ref{lojalni}(v), we get that the function $G(\cdot;\cdot)$ is Bohr $({\mathcal B}^{\circ},I',\rho)$-almost periodic ($({\mathcal B}^{\circ},I',\rho)$-uniformly recurrent). Hence, the sequence $(Q_{k}(\cdot;\cdot)=F_{k}(\cdot;\cdot)-G_{k}(\cdot;\cdot))$ converges to a function $Q(\cdot;\cdot),$ uniformly on $I\times B^{\circ}.$
This simply implies $Q\in C_{0,I,{\mathcal B}^{\circ}}(I\times X : Y),$ finishing the proof.
\end{proof}

Set $I_{{\bf t}}:=(-\infty,t_{1}] \times (-\infty,t_{2}]\times \cdot \cdot \cdot \times (-\infty,t_{n}]$ and 
${\mathbb D}_{{\bf t}}:=I_{{\bf t}} \cap {\mathbb D}$
for any ${\bf t}=(t_{1},t_{2},\cdot \cdot \cdot, t_{n})\in {\mathbb R}^{n}.$
The following result extends the statement of \cite[Proposition 2.23]{nova-mse} in the case that $X=Y;$ the proof follows from Theorem \ref{milenko} and the argumentation contained in the proof of \cite[Proposition 2.56]{marko-manuel-ap}, where we have assumed $I'=I:$

\begin{prop}\label{hmhm}
Let $A$ be a closed linear operator on $X,$ and let $(R({\bf t}))_{{\bf t}> {\bf 0}}\subseteq L(X)$ be a strongly continuous operator family such that $R({\bf t})A\subseteq AR({\bf t})$ for all ${\bf t}\in {\mathbb R}^{n}$  and
$\int_{(0,\infty)^{n}}\|R({\bf t} )\|\, d{\bf t}<\infty .$ 
Suppose, further, that $g : {\mathbb R}^{n} \rightarrow X$ is a bounded $(I',A)$-almost periodic function, resp.  bounded $(I',A)$-uniformly recurrent function, and the function $Ag : {\mathbb R}^{n} \rightarrow X$ is well defined and bounded, 
$q\in C_{0,{\mathbb D}}(I: X), $ $f:=g+q,$
\begin{align*}
\lim_{|{\bf t}|\rightarrow \infty,  {\bf t} \in {\mathbb D}}\int_{I_{{\bf t}}\cap {\mathbb D}^{c}}\| R({\bf t}-{\bf s})\|\, d{\bf s}=0
\end{align*}
and for each $r>0$ we have
\begin{align*}
\lim_{|{\bf t}|\rightarrow \infty, {\bf t} \in {\mathbb D}}\int_{{\mathbb D}_{{\bf t}}\cap B(0,r)}\| R({\bf t}-{\bf s})\|\, d{\bf s}=0.
\end{align*}
Then the function 
\begin{align*}
F({\bf t}):=\int_{{\mathbb D}_{{\bf t}}}R({\bf t}-{\bf s})f({\bf s})\, ds,\quad {\bf t}\in I
\end{align*}
is strongly ${\mathbb D}$-asymptotically $(I',A)$-almost periodic, resp. strongly ${\mathbb D}$-asymptotically $(I',A)$-uniformly recurrent; furthermore, the principal part of $F(\cdot)$ is bounded $(I',A)$-almost periodic, resp. bounded $(I',A)$-uniformly recurrent.
\end{prop}

If ${\mathbb D}=[\alpha_{1},\infty) \times [\alpha_{2},\infty) \times \cdot \cdot \cdot \times [\alpha_{n},\infty)$ for some real numbers $\alpha_{1},\ \alpha_{2},\cdot \cdot \cdot,\ \alpha_{n},$ then ${\mathbb D}_{{\bf t}}=[\alpha_{1},t_{1}]\times [\alpha_{2},t_{2}] \times \cdot \cdot \cdot \times [\alpha_{n},t_{n}].$ In this case, 
the function
$
F({\bf t})=\int^{{\bf \alpha}}_{{\bf t}}R({\bf t}-{\bf s})f({\bf s})\, ds,$ $ {\bf t}\in I
$ is strongly ${\mathbb D}$-asymptotically $(I',A)$-almost periodic, resp. strongly ${\mathbb D}$-asymptotically $(I',A)$-uniformly recurrent, where we accept the notation
$$
\int^{{\bf \alpha}}_{{\bf t}}\cdot =\int_{\alpha_{1}}^{t_{1}}\int_{\alpha_{2}}^{t_{2}}\cdot \cdot \cdot \int_{\alpha_{n}}^{t_{n}}.
$$

The following definition is also meaningful:

\begin{defn}\label{nafaks123456789012345123ceaT}
Suppose that 
${\mathbb D} \subseteq I \subseteq {\mathbb R}^{n}$ and the set ${\mathbb D}$ is unbounded, as well as
$\emptyset  \neq I'\subseteq {\mathbb R}^{n},$ $\emptyset  \neq I\subseteq {\mathbb R}^{n},$ 
$F : I \times X \rightarrow Y$ is a continuous function, $I +I' \subseteq I$ and $\rho$ is a binary relation on $X.$ Then we say that:
\begin{itemize}
\item[(i)]
$F(\cdot;\cdot)$ is ${\mathbb D}$-asymptotically Bohr $({\mathcal B},I',\rho)$-almost periodic  of type $1$ if and only if for every $B\in {\mathcal B}$ and $\epsilon>0$
there exist $l>0$ and $M>0$ such that for each ${\bf t}_{0} \in I'$ there exists ${\bf \tau} \in B({\bf t}_{0},l) \cap I'$ such that, for every ${\bf t}\in I$ and $x\in B$ with ${\bf t},\ {\bf t}+\tau \in {\mathbb D}_{M}$, there exists an element $y_{{\bf t},x}\in \rho(F({\bf t};x))$ such that
\begin{align}\label{emojmarko145ceT}
\bigl\|F({\bf t}+{\bf \tau};x)-y_{{\bf t},x}\bigr\|_{Y} \leq \epsilon .
\end{align}
\item[(ii)] 
$F(\cdot;\cdot)$ is ${\mathbb D}$-asymptotically $({\mathcal B},I',\rho)$-uniformly recurrent  of type $1$ if and only if for every $B\in {\mathcal B}$ 
there exist a sequence $({\bf \tau}_{k})$ in $I'$ and a sequence $(M_{k})$ in $(0,\infty)$ such that $\lim_{k\rightarrow +\infty} |{\bf \tau}_{k}|=\lim_{k\rightarrow +\infty}M_{k}=+\infty$ and that, for every ${\bf t}\in I$ and $x\in B,$ there exists an element $y_{{\bf t};x}\in \rho (F({\bf t};x))$ such that
$$
\lim_{k\rightarrow +\infty}\sup_{{\bf t},{\bf t}+{\bf \tau}_{k}\in {\mathbb D}_{M_{k}};x\in B} \bigl\|F({\bf t}+{\bf \tau}_{k};x)-y_{{\bf t},x}\bigr\|_{Y} =0.
$$
\end{itemize}
In the case that $X=\{0\}$ ($I'=I$), we omit the term ``${\mathcal B}$'' (``$I'$'') from the notation, as before.
\end{defn}

The proof of following slight extension of \cite[Proposition 2.26]{nova-mse} is trivial:

\begin{prop}\label{okejecew}
Suppose that 
${\mathbb D} \subseteq I \subseteq {\mathbb R}^{n}$ and the set ${\mathbb D}$ is unbounded, as well as
$\emptyset  \neq I'\subseteq {\mathbb R}^{n},$ $\emptyset  \neq I\subseteq {\mathbb R}^{n},$ 
$F : I \times X \rightarrow Y$ is a continuous function, $I +I' \subseteq I$ and $\rho$ is a binary relation on $X.$ If  
$F(\cdot;\cdot)$ is ${\mathbb D}$-asymptotically Bohr $({\mathcal B},I',\rho)$-almost periodic, resp. ${\mathbb D}$-asymptotically $({\mathcal B},I',\rho)$-uniformly recurrent, then $F(\cdot;\cdot)$ is ${\mathbb D}$-asymptotically Bohr $({\mathcal B},I',\rho)$-almost periodic of type $1,$
resp. ${\mathbb D}$-asymptotically $({\mathcal B},I',\rho)$-uniformly recurrent of type $1$.
\end{prop}

In the case that $\rho=c{\rm I},$ where $c\in {\mathbb C} \setminus \{0\}$ satisfies $|c|=1,$ the converse statement holds provided some additional conditions on the region $I'=I;$ see \cite[Theorem 2.27]{nova-mse}. This result can be further extended to the case in which $\rho=T\in L(X)$ is not necessarily a linear isomorphism; more precisely, we have the following:

\begin{thm}\label{bounded-paziemceqT}
Suppose that $\rho=T\in L(X)$, $\emptyset  \neq I \subseteq {\mathbb R}^{n},$ $I +I =I,$ $I$ is closed as well as
$F : I \rightarrow Y$ is a uniformly continuous, bounded function which is both 
$I$-asymptotically Bohr $T$-almost periodic function of type $1$
and $I$-asymptotically Bohr ${\rm I}$-almost periodic function of type $1.$
If
\begin{align*}
\notag
(\forall l>0) \, (\forall M>0) \, (\exists {\bf t_{0}}\in I)\, (\exists k>0) &\, (\forall {\bf t} \in I_{M+l})(\exists {\bf t_{0}'}\in I)\,
\\ & (\forall {\bf t_{0}''}\in B({\bf t_{0}'},l) \cap I)\, {\bf t}- {\bf t_{0}''} \in B({\bf t_{0}},kl) \cap I_{M},
\end{align*}
there exists $L>0$ such that $I_{kL}\setminus I_{(k+1)L}  \neq \emptyset$ for all $k\in {\mathbb N}$ and $I_{M}+I\subseteq  I_{M}$ for all $M>0,$
then the function $F(\cdot)$ is $I$-asymptotically Bohr $T$-almost periodic.
\end{thm}

\begin{proof}
The proof of \cite[Theorem 2.27]{nova-mse} contains some minor typographical errors and, because of that, we we will provide all details of the proof for the sake of completeness.
Since we have assumed that the function $F(\cdot)$ is $I$-asymptotically Bohr ${\rm I}$-almost periodic function of type $1,$
using \cite[Theorem 2.34]{marko-manuel-ap}
it follows that for each sequence $({\bf b}_{k})$ in  
$I$ there exist a subsequence 
$({\bf b}_{k_{l}})$ of $({\bf b}_{k})$ and a function $F^{\ast} : I \rightarrow Y$ such that $\lim_{l\rightarrow +\infty}F({\bf t}+{\bf b}_{k_{l}})=F^{\ast}({\bf t}),$ uniformly in ${\bf t}\in I.$ Clearly,
for each integer $k\in {\mathbb N}$  
there exist $l_{k}>0$ and $M_{k}>0$ such that for each ${\bf t}_{0} \in I$ there exists ${\bf \tau} \in B({\bf t}_{0},l) \cap I$ such that
\eqref{emojmarko145ceT} holds with $T={\rm I},$ $\epsilon=1/k$ and ${\mathbb D}=I.$
Let ${\bf \tau}_{k}$ be any fixed element of $I$ such that $|{\bf \tau}_{k}|>M_{k}+k^{2}$ and \eqref{emojmarko145ceT} holds with $T={\rm I},$ $\epsilon=1/k$ and ${\mathbb D}=I$ ($k\in {\mathbb N}$).
Then there exist of a subsequence $({\bf \tau}_{k_{l}})$ of $({\bf \tau}_{k})$ and a function
$F^{\ast} : I \rightarrow Y$
such that
\begin{align}\label{prce-franjo}
\lim_{l\rightarrow +\infty}F({\bf t}+{\bf \tau}_{k_{l}})=F^{\ast}({\bf t}),\mbox{ uniformly for }t\in I.
\end{align} 
The mapping $F^{\ast}(\cdot)$ is clearly continuous and we can prove that $F^{\ast}(\cdot)$ is Bohr $T$-almost periodic as follows. 
Let $\epsilon>0$ be fixed, and let $l>0$ and $M>0$ be such that for each ${\bf t}_{0} \in I$ there exists ${\bf \tau} \in B({\bf t}_{0},l) \cap I$ such that
\eqref{emojmarko145ceT} holds with ${\mathbb D}=I$ and the number $\epsilon$ replaced therein by $\epsilon/(3(1+\| T\|)).$ Let ${\bf t }\in I$ be fixed, and let $l_{0}\in {\mathbb N}$ be such that $|{\bf t}+{\bf \tau}_{k_{l_{0}}}|\geq M$ and $|{\bf t}+{\bf \tau}+{\bf \tau}_{k_{l_{0}}}|\geq M.$ Then 
\begin{align*}
\Bigl\| & F^{\ast}({\bf t}+{\bf \tau})- T F^{\ast}({\bf t})\Bigr\|
\\& \leq \Bigl\|  F^{\ast}({\bf t}+{\bf \tau})- F\bigl({\bf t}+{\bf \tau}+{\bf \tau}_{k_{l_{0}}}\bigr)\Bigr\|+\Bigl\|  F\bigl({\bf t}+{\bf \tau}+{\bf \tau}_{k_{l_{0}}}\bigr)-T F\bigl({\bf t}+{\bf \tau}_{k_{l_{0}}}\bigr)\Bigr\|
\\&+
\Bigl\|  T F\bigl({\bf t}+{\bf \tau}_{k_{l_{0}}}\bigr)- T F^{\ast}({\bf t}) \Bigr\|\leq 2\cdot (\epsilon/3)+\|T\| \cdot \Bigl\|   F\bigl({\bf t}+{\bf \tau}_{k_{l_{0}}}\bigr)- F^{\ast}({\bf t}) \Bigr\|\leq \epsilon,
\end{align*}
which implies the required. The function ${\bf t} \mapsto F({\bf t} )-F^{\ast}({\bf t} ),$ ${\bf t} \in I$ belongs to the space $C_{0,I}(I: Y)$ due to \eqref{prce-franjo} and the fact that $F : I \rightarrow Y$ is an
$I$-asymptotically Bohr ${\rm I}$-almost periodic function of type $1.$ This completes the proof.
\end{proof}

Concerning the extensions of Bohr $(I',\rho)$-almost periodic functions and $(I',\rho)$-uniformly recurrent functions, we may deduce the following result in which we require that $\rho=T\in L(X)$ is a linear isomorphism; the proof is almost the same as the corresponding proof of 
\cite[Theorem 2.28]{nova-mse} (we only need to replace any appearance of the letters $c$ and $c^{-1}$ with the letters $T$ and $T^{-1}$, respectively, in a part of proof where $\arg(c) /\pi \notin {\mathbb Q}$):

\begin{thm}\label{lenny-jassonceT}
Suppose that $\rho=T\in L(Y)$ is a linear isomorphism, $ I' \subseteq {\mathbb R}^{n},$ $ I \subseteq {\mathbb R}^{n},$
$I +I' \subseteq I,$ the set $I'$ is unbounded, $F : I  \rightarrow Y$ is a uniformly continuous, Bohr $(I',T)$-almost periodic function, resp. a uniformly continuous, $(I',T)$-uniformly recurrent function, $S\subseteq {\mathbb R}^{n}$ is bounded and the following condition holds:
\begin{itemize}
\item[(AP-E)]\index{condition!(AP-E)}
For every ${\bf t}'\in {\mathbb R}^{n},$
there exists a finite real number $M>0$ such that
${\bf t}'+I'_{M}\subseteq I.$
\end{itemize}
Then there exists a uniformly continuous, Bohr $(I' \cup S,T)$-almost periodic, resp. a uniformly continuous, $(I' \cup S,T)$-uniformly recurrent, function
$\tilde{F} : {\mathbb R}^{n}  \rightarrow Y$ such that $\tilde{F}({\bf t})=F({\bf t})$ for all ${\bf t}\in I;$ furthermore, in $T$-almost periodic case, the uniqueness of such a function $\tilde{F}(\cdot)$ holds provided that ${\mathbb R}^{n} \setminus (I' \cup S)$ is a bounded set and any Bohr $T$-almost periodic function defined on ${\mathbb R}^{n}$ 
is almost automorphic. 
\end{thm}

The general requirements of Theorem \ref{lenny-jassonceT} hold provided that $({\bf v}_{1},\cdot \cdot \cdot ,{\bf v}_{n})$ is a basis of ${\mathbb R}^{n}$ and
$$
I'=I=\bigl\{ \alpha_{1} {\bf v}_{1} +\cdot \cdot \cdot +\alpha_{n}{\bf v}_{n}  : \alpha_{i} \geq 0\mbox{ for all }i\in {\mathbb N}_{n} \bigr\}
$$ 
is a convex polyhedral in ${\mathbb R}^{n}.$ But, in this case, we do not have that ${\mathbb R}^{n} \setminus (I' \cup S)$ is a bounded set and the question of uniqueness of uniformly continuous, Bohr $(I',T)$-almost periodic extensions of function $F(\cdot)$ to ${\mathbb R}^{n}$ naturally appears. 

Keeping in mind Corollary \ref{rtanj1} and Theorem \ref{lenny-jassonceT}, we can immediately clarify the following result:

\begin{cor}\label{why}
Suppose that $\rho=T\in L(Y)$ is a linear isomorphism, $ I' \subseteq {\mathbb R}^{n},$ $ [0,\infty)^{n} \subseteq I \subseteq {\mathbb R}^{n},$  
$I'- I' ={\mathbb R}^{n},$ $I +I' \subseteq I,$
$F : I  \rightarrow Y$ is a uniformly continuous, Bohr $(I',T)$-almost periodic function, and condition \emph{(AP-E)} holds.
Then the mean value
${\mathcal M}_{\lambda}(F),$ given by the second equality in \eqref{idiote}, exists
for any $\lambda \in {\mathbb R}^{n}$, and the set of all points $\lambda \in {\mathbb R}^{n}$ for which ${\mathcal M}_{\lambda}(F)\neq 0$ is at most countable.
\end{cor}

It is clear that the requirements of Corollary \ref{why} holds provided that $I'=I=[0,\infty)^{n}.$ Concerning the existence of mean value, we would like to propose the following problem at the end of this section:\vspace{0.1cm}

{\sc Problem.} Suppose that $\rho=T\in L(Y)$ is not a linear isomorphism, $ I' \subseteq {\mathbb R}^{n},$ $ [0,\infty)^{n} \subseteq I \subseteq {\mathbb R}^{n},$  
$I'- I' ={\mathbb R}^{n},$ $I +I' \subseteq I,$ and condition (AP-E) holds. Does there exists a (uniformly continuous) Bohr $(I',T)$-almost periodic function
$F : I  \rightarrow Y$ such that the mean value
${\mathcal M}_{\lambda}(F),$ given by the second equality in \eqref{idiote}, does not exist for some
$\lambda \in {\mathbb R}^{n}?$\vspace{0.1cm}

Observe only that the function $F : [0,\infty) \rightarrow {\mathbb C}^{2},$ constructed in Example \ref{tomqqa}, is asymptotically almost periodic and therefore the mean value ${\mathcal M}_{\lambda}(F)$ exists for any $\lambda \in {\mathbb R}^{n}.$

\section{$({\bf \omega},\rho)$-Periodic functions and $({\bf \omega}_{j},\rho_{j})_{j\in {\mathbb N}_{n}}$-periodic functions}\label{krizni-stab}

In \cite{nds-2021}, we have recently generalized the notion of Bloch $({\bf p},{\bf k})$-periodicity by proposing a new definition of $({\bf \omega},c)$-periodicity, where ${\bf \omega}\in {\mathbb R}^{n} \setminus \{0\}$ and $c\in {\mathbb C} \setminus \{0\}.$ The notion of $({\bf \omega},c)$-periodicity is a special case of the notion $({\bf \omega},\rho)$-periodicity introduced as follows:

\begin{defn}\label{drasko-presing}
Let ${\bf \omega}\in {\mathbb R}^{n} \setminus \{0\},$ $\rho$ be a binary relation on $X$ 
and 
${\bf \omega}+I \subseteq I$. A continuous
function $F:I\rightarrow X$ is said to be $({\bf \omega},\rho)$-periodic if and only if 
$
F({\bf t}+{\bf \omega})\in \rho(F({\bf t})),$ ${\bf t}\in I.
$ 
\end{defn}

\begin{example}\label{visevr}
Let $C\in L(X)$, and let $x\in X.$ A strongly continuous family $(T(t))_{t\geq 0}$ in $L(X)$
is said to be a $C$-regularized semigroup
if and only if
$T(t+s)C=T(t)T(s)$ for all $t$, $s \geq 0$ and $T(0)=C.$ 
Put $u(t):=T(t)x,$ $t\geq 0.$ Then the function $u(\cdot)$ is a mild solution of the abstract Cauchy inclusion $u^{\prime}(t)\in {\mathcal A}u(t),$ $t\geq 0;$ $u(0)=Cx,$ where a closed MLO ${\mathcal A}$ is the integral generator of $(T(t))_{t\geq 0};$ see \cite[Section 3.1-Section 3.4]{FKP} for more details on the subject.
It is clear that $u(\cdot)$ is $(\omega,C^{-1}T(\omega))$-periodic as well as that the binary relation $\rho_{\omega}:=C^{-1}T(\omega)$ is not single-valued in general ($\omega>0$).
\end{example}

We have the following simple result: 

\begin{prop}\label{zasu}
Suppose that  ${\bf \omega}\in {\mathbb R}^{n} \setminus \{0\},$ $\rho$ is a binary relation on $X,$  
${\bf \omega}+I \subseteq I$ and the function
$F:I\rightarrow X$ is $({\bf \omega},\rho)$-periodic. If $A$ and $B$ are syndetic subsets of ${\mathbb N},$ we set
$\sigma:=\bigcup_{m\in A}\rho^{m}$ and $I':=\{m\omega : m\in B\}.$ 
Then the function $F(\cdot)$ is Bohr $(I',\sigma)$-almost periodic.
\end{prop}

\begin{proof}
Since ${\bf \omega}+I \subseteq I$, we inductively get $I+I'\subseteq I.$ Since
$
F({\bf t}+m{\bf \omega})\in \rho^{m}(F({\bf t})),$ ${\bf t}\in I,
$ $m\in {\mathbb N},$ the result immediately follows from the corresponding definition of  Bohr $(I',\sigma)$-almost periodicity, the definition of a syndetic set and the definition of binary relation $\sigma.$ 
\end{proof}

If we additionally assume that $k \omega +I\subseteq I$ for all positive real numbers $k>0,$ then the above statement continues to hold with the set $I':=[0,\infty) \cdot \omega$. In connection with the statement of Proposition \ref{zasu}, it is worth mentioning
case in which we have the existence of a positive integer $m\in {\mathbb N}$ such that
$\rho^{m}\subseteq \rho,$ which implies that $\rho^{km}\subseteq \rho$ 
for all $k\in {\mathbb N}.$ Then we can simply show that the preassumptions of Proposition \ref{zasu} imply that the function $F(\cdot;\cdot)$ is Bohr $(I',\rho)$-almost periodic (with the set $I':=[0,\infty) \cdot \omega$, provided that $k \omega +I\subseteq I$ for all positive real numbers $k>0$). It is worth noting that the choice of binary relation $\sigma$ is important here because, if $c\in {\mathbb C} \setminus \{0\}$ and $|c|\neq 1,$ then the $(\omega,c)$-periodicity of a non-zero function $F(\cdot)$ cannot imply its $c$-almost periodicity due to \cite[Proposition 2.7]{c1} ($\rho=c{\rm I}$).

The following definition is also meaningful:

\begin{defn}\label{drasko-presing1}
Let ${\bf \omega}_{j}\in {\mathbb R} \setminus \{0\},$ $\rho_{j}\in {\mathbb C} \setminus \{0\}$ is a binary relation on $X$
and 
${\bf \omega}_{j}e_{j}+I \subseteq I$ ($1\leq j\leq n$). A continuous
function $F:I\rightarrow X$ is said to be $({\bf \omega}_{j},\rho_{j})_{j\in {\mathbb N}_{n}}$-periodic if and only if 
$
F({\bf t}+{\bf \omega}_{j}e_{j})\in \rho_{j}(F({\bf t})),$ ${\bf t}\in I,
$ $j\in {\mathbb N}_{n}.$ 
\end{defn} \index{function!$({\bf \omega}_{j},\rho_{j})_{j\in {\mathbb N}_{n}}$-periodic}

In the case that $\rho_{j}=c_{j}{\rm I}$ for some non-zero complex numbers $c_{j}$ ($1\leq j\leq n$), then we also say that the function $F(\cdot)$ is 
$({\bf \omega}_{j},c_{j})_{j\in {\mathbb N}_{n}}$-periodic; furthermore, if $c_{j}=1$  for all $j\in {\mathbb N}_{n},$ then we say that $F(\cdot)$ is $({\bf \omega}_{j})_{j\in {\mathbb N}_{n}}$-periodic.

It is clear that, if $F:I\rightarrow X$ is $({\bf \omega}_{j},\rho_{j})_{j\in {\mathbb N}_{n}}$-periodic, then $
F({\bf t}+m{\bf \omega}_{j}e_{j})\in \rho_{j}^{m}(F({\bf t})),$ ${\bf t}\in I,
$ $m\in {\mathbb N},$ $j\in {\mathbb N}_{n}.$ This enables one to transfer the statement of Proposition \ref{zasu} and conclusions given in the paragraph following this proposition to $({\bf \omega}_{j},\rho_{j})_{j\in {\mathbb N}_{n}}$-periodic functions; details can be omitted.

The interested reader may try to reformulate the statement of Theorem \ref{lojalni} for $({\bf \omega},\rho)$-periodic functions and $({\bf \omega}_{j},\rho_{j})_{j\in {\mathbb N}_{n}}$-periodic functions. Further on, in the scalar-valued case, the following holds: If the function $F : I\rightarrow {\mathbb C} \setminus \{0\}$ is $({\bf \omega},\rho)$-periodic, resp. $({\bf \omega}_{j},\rho_{j})_{j\in {\mathbb N}_{n}}$-periodic, then the function $(1/F)(\cdot)$ is $({\bf \omega},\sigma)$-periodic, resp. $({\bf \omega}_{j},\sigma_{j})_{j\in {\mathbb N}_{n}}$-periodic, provided that for every non-zero complex numbers $x$ and $y,$ we have that the assumption $(x,y)\in \rho$ implies $(1/x,1/y)\in \sigma,$ resp. the assumption $(x,y)\in \rho_{j}$ implies $(1/x,1/y)\in \sigma_{j}$ for all $j\in {\mathbb N}_{n}.$ This can be also reworded for Bohr $({\mathcal B},I',\rho)$-almost periodic functions and $({\mathcal B},I',\rho)$-uniformly recurrent functions; details can be left to the interested readers.

It is straightforward to deduce the following extension of \cite[Proposition 2.5]{nds-2021}:

\begin{prop}\label{stanivuko}
Let ${\bf \omega}_{j}\in {\mathbb R} \setminus \{0\},$ $\rho_{j}$ is a binary relation on $X$
and 
${\bf \omega}_{j}e_{j}+I \subseteq I$ ($1\leq j\leq n$). If a continuous
function $F:I\rightarrow X$ is $({\bf \omega}_{j},\rho_{j})_{j\in {\mathbb N}_{n}}$-periodic and $\sigma : {\mathbb N}_{n} \rightarrow {\mathbb N}_{n}$ is a permutation, then
$\omega +I \subseteq I,$ where $\omega:=\sum_{j=1}^{n}\omega_{j}e_{j},$ and
the function $F(\cdot)$ is  $({\bf \omega},\rho)$-periodic with $\rho=:\prod_{j=1}^{n}\rho_{\sigma(j)}.$
\end{prop}

We continue with the following illustrative example:

\begin{example}\label{deo-javnosti}
As we have shown in \cite[Example 2.2, Example 2.6]{nds-2021}, 
there exists a continuous, unbounded function $F: {\mathbb R}^{n} \rightarrow {\mathbb R}$ which satisfies that $F({\bf t}+(1,1,\cdot \cdot \cdot,1))=F({\bf t})$ for all ${\bf t}\in {\mathbb R}^{n}$ as well as that there do not exist numbers 
$\omega_{1},\ \omega_{2} \in {\mathbb R} \setminus \{0\}$ and numbers $c_{1}, \ c_{2}\in {\mathbb C} \setminus \{0\}$ such that the function $F(\cdot)$ is $({\bf \omega}_{j},c_{j})_{j\in {\mathbb N}_{2}}$-periodic.
The analysis carried out in these examples shows that there do not exist numbers 
$\omega_{1},\ \omega_{2} \in {\mathbb R} \setminus \{0\},$ a number $c_{1}=\rho_{1}\in {\mathbb C} \setminus \{0\}$ and a binary relation $\rho_{2}$ on ${\mathbb C}$ such that the function $F(\cdot)$ is $({\bf \omega}_{j},\rho_{j})_{j\in {\mathbb N}_{2}}$-periodic.
\end{example}

The reader can simply formulate certain statements concerning
the pointwise multiplication of scalar-valued $({\bf \omega},\rho)$-periodic functions ($({\bf \omega}_{j},\rho_{j})_{j\in {\mathbb N}_{n}}$-periodic functions) and the vector-valued $({\bf \omega},\rho)$-periodic functions ($({\bf \omega}_{j},\rho_{j})_{j\in {\mathbb N}_{n}}$-periodic functions).
Concerning the convolution invariance of spaces consisting of $({\bf \omega},\rho)$-periodic functions ($({\bf \omega}_{j},\rho_{j})_{j\in {\mathbb N}_{n}}$-periodic functions), we will state and prove the following result:

\begin{prop}\label{debilpropo}
Suppose that ${\bf \omega}\in {\mathbb R}^{n} \setminus \{0\},$ $\rho={\mathcal A}$ is a closed \emph{MLO} on $X,$
and $\rho_{j}={\mathcal A}_{j}$ is a closed \emph{MLO} on $X$ ($1\leq j\leq n$). 
Suppose, further, that  $F : {\mathbb R}^{n} \rightarrow X$ is $({\bf \omega},{\mathcal A})$-periodic ($({\bf \omega}_{j},{\mathcal A}_{j})_{j\in {\mathbb N}_{n}}$-periodic) and bounded. If $h\in L^{1}({\mathbb R}^{n}),$ then the function
$$
(h\ast F)({\bf t}):=\int_{{\mathbb R}^{n}}h(y)F({\bf t}-y)\, dy,\quad {\bf t}\in {\mathbb R}^{n}
$$
is $({\bf \omega},{\mathcal A})$-periodic ($({\bf \omega}_{j},{\mathcal A}_{j})_{j\in {\mathbb N}_{n}}$-periodic) and bounded.
\end{prop}

\begin{proof}
Keeping in mind the dominated convergence theorem, the boundedness and the continuity of function $F(\cdot),$ we easily get that the function $
(h\ast F)(\cdot)$ is well defined, bounded and continuous. Suppose that the function $F(\cdot)$ is
$({\bf \omega},{\mathcal A})$-periodic. Then we have $F({\bf t}+\omega-{\bf s})\in {\mathcal A}(F({\bf t}-{\bf s}))$ and therefore
$h({\bf s})F({\bf t}+\omega-{\bf s})\in {\mathcal A}(h({\bf s})F({\bf t}-{\bf s}))$ (${\bf t},\ {\bf s}\in {\mathbb R}^{n}$). Applying Lemma \ref{stana}, we get $(h\ast F)({\bf t}+\omega) \in {\mathcal A}((h\ast F)({\bf t}))$ for all ${\bf t}\in {\mathbb R}^{n},$ as required. The proof is same for 
$({\bf \omega}_{j},{\mathcal A}_{j})_{j\in {\mathbb N}_{n}}$-periodic functions.
\end{proof}

We leave to the interested readers problem of transferring the statements of Theorem \ref{mkmk} and Theorem \ref{episkop-jovan} for $({\bf \omega},\rho)$-periodic functions and $({\bf \omega}_{j},\rho_{j})_{j\in {\mathbb N}_{n}}$-periodic functions.
For simplicity, we will not consider here ${\mathbb D}$-asymptotically $({\bf \omega},\rho)$-periodic functions and ${\mathbb D}$-asymptotically $({\bf \omega}_{j},\rho_{j})_{j\in {\mathbb N}_{n}}$-periodic functions, as well.

\section{Applications}\label{castel}

In this section, we will present several important examples and applications of our results to the abstract Volterra integro-differential equations in Banach spaces. 

\subsection{Liouville type results for anti-periodic linear
operators}\label{ljuvil}

Our first applications are closely connected with the investigation of L. Rossi \cite{luca-rossi}, which concerns
Liouville type results for periodic and almost periodic linear
operators. The author has investigated bounded solutions of the partial differential equation
\begin{align}\label{nadrini}
Pu:={\partial}_{t}u-a_{ij}(x,t){\partial_{ij}}u-b_{i}(x,t){\partial _{i}}u-c(x,t)u=f(x,t),\quad x\in {\mathbb R}^{n},\ t\in {\mathbb R};
\end{align}
all considered coefficients $a_{ij}(x,t),$ $b_{i}(x,t)$, $c(x,t)$ are assumed to be real-valued and satisfy the general assumptions stated on \cite[p. 2482]{luca-rossi}. 

If the coefficients $a_{ij}(x,t),$ $b_{i}(x,t)$, $c(x,t)$
do not depend on $t$, then $P$ can be written as $Pu={\partial}_{t}u-Lu$, where $L$ is a linear differential operator in ${\mathbb R}^{n}.$ Furthermore, if the coefficients 
$a_{ij}(x),$ $b_{i}(x)$, $c(x)$ are $(\omega_{j})_{j\in {\mathbb N}_{n}}$-periodic (${\bf \omega}_{j}\in {\mathbb R} \setminus \{0\}$ for all $j\in {\mathbb N}_{n}$), then we say that $L$ is $(\omega_{j})_{j\in {\mathbb N}_{n}}$-periodic, as well.
If this is the case, then some results from the
Krein Rutman theory yield the existence of a unique real number $\lambda$, called periodic principal eigenvalue of $-L,$
such that the eigenvalue problem
$$-L\phi =\lambda \phi \mbox{ in }{\mathbb R}^{n};\ \
\phi(\cdot) \mbox{ is periodic, with the same period as } L,
$$
admits positive solutions. It is well known that the positive solution $\phi(\cdot),$ which is called the principal eigenfunction, is unique up to a multiplicative constant; by $\lambda_{p}(-L)$ and $\varphi_{p}(\cdot)$ we denote the periodic principal eigenvalue
and eigenfunction of $-L,$ respectively.

In general case, the coefficients $a_{ij}(x,t),$ $b_{i}(x,t)$, $c(x,t)$ can be complex-valued. Then it is said that the linear differential operator $P$ is:
\begin{itemize}
\item[(i)] $({\bf \omega},c)$-periodic (${\bf \omega}\in {\mathbb R}^{n+1} \setminus \{0\},$ $c\in {\mathbb C} \setminus \{0\}$) if and only if 
all functions $a_{ij}(x,t),$ $b_{i}(x,t)$, $c(x,t)$ are  $({\bf \omega},c)$-periodic;
\item[(ii)] $({\bf \omega}_{j},c_{j})_{j\in {\mathbb N}_{n+1}}$-periodic (${\bf \omega}_{j}\in {\mathbb R} \setminus \{0\},$ $c_{j}\in {\mathbb C} \setminus \{0\}$) if and only if 
all functions $a_{ij}(x,t),$ $b_{i}(x,t)$, $c(x,t)$ are
$({\bf \omega}_{j},c_{j})_{j\in {\mathbb N}_{n+1}}$-periodic.
\end{itemize}

In the case that all coefficients $a_{ij}(x,t),$ $b_{i}(x,t)$, $c(x,t)$ are real-valued, we can slightly extend the statement of \cite[Lemma 2.1, Theorem 1.1, Theorem 1.3(i)]{luca-rossi} in the following way, by considering (anti-)periodic case $c=\pm 1$ in place of the already considered periodic case $c=1$ (with the meaning clear):

\begin{lem}\label{dzez}
Suppose that the linear differential operator $P$ is $(l_{m}e_{m},1)$-periodic and the function $f(x,t)$ is $(l_{m}e_{m},c)$-periodic, where $c=\pm 1,$ $l_{m}\in {\mathbb R} \setminus \{0\}$ and $m\in {\mathbb N}_{n+1}.$ If there exists a bounded continuous function $v(x,t)$ satisfying that $\inf_{(x,t)\in {\mathbb R}^{n+1}}v(x,t)>0$ and $Pv=\phi$ for some nonnegative function $\phi \in L^{\infty}({\mathbb R}^{n+1}),$  then any bounded solution $u(x,t)$ of the equation \eqref{nadrini} is $(l_{m}e_{m},c)$-periodic.
\end{lem} 

\begin{proof}
The proof is almost completely the same as the proof of \cite[Lemma 2.1]{luca-rossi} and we will only point out  a few important details for case $c=-1.$ First of all, we consider the function
$\psi(X):=u(X+l_{m}e_{m})+u(X),$ $X\in {\mathbb R}^{n+1}$ in place of the already considered function $\psi(X):=u(X+l_{m}e_{m})-u(X),$ $X\in {\mathbb R}^{n+1}.$ We will prove that
$\psi(X) \leq 0$ for all $X\in {\mathbb R}^{n+1}.$ If we assume that this is not true, then there exists a sequence $(X_{N})$ in ${\mathbb R}^{n+1}$ such that
$\psi(X_{N})/v(X_{N}) \rightarrow k\equiv \sup_{X\in {\mathbb R}^{n+1}} (\psi/v)(X)>0$ as $N\rightarrow +\infty.$ Define $\psi_{N}:=\psi(\cdot+X_{N}),$ $N\in {\mathbb N}.$ Since we have assumed that the linear differential operator $P$ is $(l_{m}e_{m},1)$-periodic and the function $f(x,t)$ is $(l_{m}e_{m},-1)$-periodic, a simple calculation shows that
$$
{\partial}_{t}\psi_{N}-a_{ij}\bigl(X+X_{N}\bigr){\partial_{ij}}\psi_{N}-b_{i}\bigl(X+X_{N}\bigr){\partial _{i}}\psi_{N}-c\bigl(X+X_{N}\bigr)\psi_{N}=0,\ X\in {\mathbb R}^{n+1}, \ N\in {\mathbb N}.
$$
Then we can repeat verbatim all calculations and arguments from the corresponding proof given on p. 2486 of \cite{luca-rossi}; at the last line of this page,
we only need to write $\xi_{h}+\xi_{h+1}$ in place of $\xi_{h}-\xi_{h+1}$
 in order to conclude that $\lim_{h\rightarrow +\infty}[\xi_{h}+\xi_{h+1}]=-\infty,$ which is a contradiction with the boundedness of sequence $(\xi_{h})_{h\in {\mathbb N}}.$ This implies $u(X+l_{m}e_{m})+u(X)\leq 0$ for all $X\in {\mathbb R}^{n+1}.$ It is our strong belief that
the final part of the proof of the above-mentioned lemma is a little bit incorrect as well as that we can use a simple trick here, working also in the case that $c=-1:$ since $P[-u]=-f,$ we can apply the first part of proof (with the same number $l_{m};$ see also \cite[p. 2487, l. 2]{luca-rossi}, where the author suggests the use of number $-l_{m}$) in order to see that 
$u(X+l_{m}e_{m})+u(X)\geq 0$ for all $X\in {\mathbb R}^{n+1}.$ Therefore, $u(X+l_{m}e_{m})+u(X)= 0$ for all $X\in {\mathbb R}^{n+1},$ which completes the proof.
\end{proof}

\begin{thm}\label{bluz}
Suppose that the linear differential operator $P$ is $(l_{m}e_{m},1)$-periodic and the function $f(x,t)$ is $(l_{m}e_{m},c)$-periodic, where $c=\pm 1,$ $l_{m}\in {\mathbb R} \setminus \{0\}$ and $m\in {\mathbb N}_{n+1}.$ Then any bounded solution $u(x,t)$ of the equation \eqref{nadrini} is $(l_{m}e_{m},c)$-periodic, provided that $c(x,t)\leq 0$ for all $x\in {\mathbb R}^{n}$ and $t\in {\mathbb R}.$
\end{thm}

\begin{proof}
It suffices to put $c(x,t)\equiv 1$ in Lemma \ref{dzez}.
\end{proof}

\begin{thm}\label{bluz1}
Suppose that the linear differential operator $L$ is $(\omega_{j})_{j\in {\mathbb N}_{n}}$-periodic (${\bf \omega}_{j}\in {\mathbb R} \setminus \{0\}$ for all $j\in {\mathbb N}_{n}$), and $P=\partial_{t}-L.$
If the function 
$f(\cdot)$ is $({\bf \omega}_{j},c_{j})_{j\in {\mathbb N}_{n+1}}$-periodic with some $\omega_{n+1}\in {\mathbb R} \setminus \{0\}$ and $c_{j}=\pm 1$ for all $j\in {\mathbb N}_{n+1},$ 
then any bounded solution $u(x,t)$ of the equation \eqref{nadrini} is $({\bf \omega}_{j},c_{j})_{j\in {\mathbb N}_{n+1}}$-periodic.
\end{thm}

\begin{proof}
The proof simply follows from Lemma \ref{dzez} and the corresponding part of proof of \cite[Theorem 1.3(i)]{luca-rossi}.
\end{proof}

We continue by observing that the statements of \cite[Theorem 1.3(ii)-(iii)]{luca-rossi} do not admit satisfactory reformulations for multi-dimensional almost anti-periodic type functions as well as that it would be very tempting to reconsider the statements of \cite[Theorem 1.3(i); Theorem 4.1]{luca-rossi} for multi-dimensional $c$-almost periodic type functions. For the sake of simplicity, we will not consider here the corresponding analogues of Lemma \ref{dzez} and Theorem \ref{bluz}-Theorem \ref{bluz1} for the Dirichlet and Robin boundary value problems on uniformly smooth domains in ${\mathbb R}^{n}$ (see \cite[Section 5]{luca-rossi} for more details).

\subsection{Applications to the heat equation in ${\mathbb R}^{n}$}

Our results on the convolution invariance of multi-dimensional $\rho$-almost periodicity can be applied to the Gaussian semigroup in ${\mathbb R}^{n}$ and the Poisson semigroup in ${\mathbb R}^{n}$ without any difficulties. The obtained results can be further generalized for general binary relations by assuming some extra conditions and we will briefly explain this idea here.

Let $Y:=BUC({\mathbb R}^{n})$ be the Banach space of bounded uniformly continuous functions $F : {\mathbb R}^{n} \rightarrow {\mathbb C}.$
Then it is well known that the Gaussian semigroup\index{Gaussian semigroup}
$$
(G(t)F)(x):=\bigl( 4\pi t \bigr)^{-(n/2)}\int_{{\mathbb R}^{n}}F(x-y)e^{-\frac{|y|^{2}}{4t}}\, dy,\quad t>0,\ f\in Y,\ x\in {\mathbb R}^{n},
$$
can be extended to a bounded analytic $C_{0}$-semigroup of angle $\pi/2,$ generated by the Laplacian $\Delta_{Y}$ acting with its maximal distributional domain in $Y;$ see e.g. \cite[Example 3.7.6]{a43} for more details. Suppose 
that 
$\emptyset  \neq I' \subseteq {\mathbb R}^{n},$ $\emptyset  \neq I\subseteq {\mathbb R}^{n},$ and
$F(\cdot)$ is bounded Bohr $({\mathcal B},I',\rho)$-almost periodic, resp. bounded $({\mathcal B},I',\rho)$-uniformly recurrent, where
$\rho$ is any binary relation defined on ${\mathbb C}$ such that
for each ${\bf t}\in {\mathbb R}^{n}$ the element $y_{{\bf t}}=z_{{\bf t}} \in \rho(F({\bf t}))$ is chosen such that the equation \eqref{oblak} holds as well as that the mapping ${\bf s} \mapsto z_{{\bf t}-{\bf s}},$ ${\bf s}\in {\mathbb R}^{n}$ belongs to $Y.$ 
Let a number $t_{0}>0$ be fixed.
We define a new binary relation $\sigma$ on ${\mathbb C}$ by
$$
\sigma :=\Biggl\{ \Biggl((G(t_{0})F)(x), \bigl( 4\pi t_{0} \bigr)^{-(n/2)}\int_{{\mathbb R}^{n}}z_{x-y}e^{-\frac{|y|^{2}}{4t}}\, dy \Biggr) : x\in {\mathbb R}^{n} \Biggr\}.
$$
Then 
the function ${\mathbb R}^{n}\ni x\mapsto u(x,t_{0})\equiv (G(t_{0})F)(x) \in {\mathbb C}$
is bounded Bohr $({\mathcal B},I',\sigma)$-almost periodic, resp. bounded $({\mathcal B},I',\sigma)$-uniformly recurrent. This holds because, for every $x,\,\ \tau \in {\mathbb R}^{n},$ we have:
\begin{align*}
\Bigl|u\bigl(x+\tau,t_{0}\bigr)-\sigma \Bigl(u\bigl(x,t_{0}\bigr)\Bigr)\Bigr| & \leq \bigl( 4\pi t_{0} \bigr)^{-(n/2)}\int_{{\mathbb R}^{n}}|F(x-y+\tau)-z_{x-y}|e^{-\frac{|y|^{2}}{4t_{0}}}\, dy
\\& \leq \bigl( 4\pi t_{0} \bigr)^{-(n/2)}\epsilon \int_{{\mathbb R}^{n}}e^{-\frac{|y|^{2}}{4t_{0}}}\, dy.
\end{align*}
For example, we can consider here the non-linear mapping $\rho(z):=z^{k},$ $z\in {\mathbb C}$ ($k\in {\mathbb N} \setminus \{1\}$). 

We can  similarly clarify the corresponding results for the Poisson semigroup, which is given by
$$
(T(t)F)(x):=\frac{\Gamma((n + 1)/2)}{\pi^{(n+1)/2}}
\int_{{\mathbb R}^{n}}F(x-y)\frac{t \cdot dy}{(t^{2}+|y|^{2})^{(n+1)/2}},\quad t>0,\ f\in Y,\ x\in {\mathbb R}^{n};
$$
see
\cite[Example 3.7.9]{a43} for more details.

\subsection{New applications in one-dimensional setting}

The trick employed in the previous subsection can be applied to the remarkable examples \cite[Examples
4, 5, 7, 8; pp. 32-34]{30} given in the research monograph of S. Zaidman.
Further on, Theorem \ref{mkmk} and Proposition \ref{hmhm} can be applied in the qualitative analysis of $\rho$-almost periodic solutions for various classes of abstract (degenerate) Volterra integro-differential equations; in the analysis of corresponding semilinear Cauchy problems, applications of  Theorem \ref{episkop-jovan} are necessary.  

For example, we can incorporate the above mentioned results
in the study of existence and uniqueness of asymptotically $\rho$-almost periodic type
solutions of the fractional Poisson heat equation\index{equation!semilinear Poisson heat}
\[\left\{
\begin{array}{l}
{\mathbf D}_{t}^{\gamma}[m(x)v(t,x)]=(\Delta -b )v(t,x) +f(t,x),\quad t\geq 0,\ x\in {\Omega};\\
v(t,x)=0,\quad (t,x)\in [0,\infty) \times \partial \Omega ,\\
 m(x)v(0,x)=u_{0}(x),\quad x\in {\Omega},
\end{array}
\right.
\]
and the following
fractional Poisson heat equation with Weyl-Liouville derivatives:
\[
D^{\gamma}_{t,+}[m(x)v(t,x)]=(\Delta -b )v(t,x) +f(t,x),\quad t\in {\mathbb R},\ x\in {\Omega};
\]
 see also \cite[Example 3.10.4]{nova-mono} for some possible applications to the differential operators in H\"older spaces.
We will omit all related details because the considerations are almost the same as in \cite{nova-mono} and a great number of other research articles.

\subsection{Heteroclinic period blow-up}\label{sec3}

We first consider a singularly perturbed symmetric ODE with a symmetric
heteroclinic cycle connecting hyperbolic equilibria. Then we deal with a
conservative ODE with the same heteroclinic structure. We show in both
cases there are $(Q,T)$ affine-periodic solutions accumulating
on the heteroclinic cycle with $T\to\infty$; this is called as
heteroclinic/homoclinic period blow-up. For instance, the
Duffing equation $\ddot x-x+2x^3=0$ has a symmetric homoclinic
cycle $\pm \tilde{ \gamma}(t)$ with $\tilde{ \gamma}(t)=(\gamma(t),\dot\gamma(t))$,
$\gamma(t)=\textrm{sech}\, t$ which is accumulated by periodic
solutions with periods tending to infinity. Related results are investigated in \cite {W,WF}. Here, $Q : {\mathbb R}^{n}\to {\mathbb R}^{n}$ is a linear mapping such that $Q^k={\rm I}$ for some $k\in{\mathbb N}$.\vspace{0.1cm}

1. {\bf Symmetric singular systems.}
Consider the system
\begin{equation}\label{3.1}
\ep \dot u\,=f(u)+\ep g(t,u,\ep )\equiv f_\ep (u,t)\,
,\end{equation} for $f\in C^2(\R^n :\R^n)$, $g\in C^2(\R\times\R^n\times\R)$ and $g(t,u,\epsilon )$ is $1$-periodic
in $t$. Suppose $f_\ep(Qu,t)=Qf_\ep(u,t)$ for any $(t,u,\ep)$.
Then $Qf_\ep(u,t)=f_\ep(Qu,t+1) $. Assume
that the boundary layer problem
$$
\dot u=f(u)
$$
has a solution $\gamma (t)$ heteroclinic to hyperbolic equilibria $p_0,\, p_1$ satisfying $Qp_0=p_1$, i.e. $\ga(t)\to p_0$ as
$t\to-\infty$ and $\ga(t)\to p_1$ as $t\to +\infty$. Furthermore,
assume that the variational system
\begin{equation}\label{3.2}
\dot u(t)=Df(\gamma (t))u(t)
\end{equation}
has the unique bounded solution $\dot \gamma (t)$ on $\R$ up to a
multiplicative constant. Then the system
adjoint to \eqref{3.2} is of the form
$$
\dot v(t)=-Df(\gamma (t))^*v(t)
$$
and has unique (up to a multiplicative constant) bounded solution
$\psi (t)\not =0$ on $\R$.

Assume \cite{BF} that a function
$$
M(\alpha )=\int_{-\infty }^{+\infty }\psi ^{*}(t)g(\alpha ,\gamma
(t),0)\, dt
$$
has a simple zero at $\alpha =\alpha _0$.
Then there exists $\ep_0>0$ such that, for any $\ep\in
(0,\ep_0)$ and $m\ge1$, the system \eqref{3.1} has an $(m,Q)$ affine-periodic solution; all such solutions accumulate on the symmetric heteroclinic cycle
$$\bigcup_{i=1}^p\overline {\left\{Q^i\gamma (t)\mid t\in
\R\right\} }.$$

As an example, we consider
\begin{equation}\label{6.1}
\ep \dot z_1=z_2,\quad \ep \dot z_2=-\sin z_1+\ep h(t)z_2,
\end{equation}
where $h : \R\to \R$ is $1$-periodic and $C^2$-smooth. Clearly,
\eqref{6.1} is equivalent to the equation
$$
\ep^2\ddot x+\sin x -\ep^2h(t)\dot x=0 .
$$
The boundary layer equation is the pendulum equation
\begin{equation}\label{66.1} \dot z_1=z_2,\quad \dot z_2=-\sin z_1 .
\end{equation}
Now we take $Q=-{\rm I}$. It is well-known that two
hyperbolic equilibria $p_0=(-\pi ,0)$ and $p_1=(\pi ,0)$ of
\eqref{66.1} are connected with a heteroclinic solution $\gamma
(t)=\big (\theta (t),\dot \theta (t)\big ),$ where $\theta (t)=\pi
-4\arctan e^{-t}$. If there is an $\alpha _0$ such that
$h(\alpha _0)=0\ne h'(\alpha _0)$, then for any $m\in \N$ and
$\ep>0$ small, the system \eqref{6.1} has a $2m$-periodic
solution $\big (z_{1m\epsilon },z_{2m\epsilon }\big );$ all such solutions accumulate on the
heteroclinic cycle $$\overline {\left\{\gamma (t)\mid t\in
\R\right\} }\bigcup \overline {\left\{-\gamma (t)\mid t\in
\R\right\} }$$ and satisfy
$$
-z_{1m\ep}(t)=z_{1m\ep}(t+m),\quad -z_{2m\ep}(t)=z_{2m\ep}(t+m)
.
$$

2. {\bf Symmetric conservative systems.}
We suppose that \begin{equation}\label{1.2}\dot
x=g(x)\end{equation} is $C^1$-smooth
and conservative, i.e., \eqref{1.2} has a $C^1$-smooth first
integral $H : \R^n\to \R$ with $H(Qu)=H(u)$ for any $u\in\R^n$. We
also suppose that \eqref{1.2} has a heteroclinic orbit $\gamma
(t)$ to hyperbolic equilibria $p_0,\, p_1$  satisfying $Qp_0=p_1$.

Let us assume that the variational system $\dot v=Dg(\ga(t))v$ has the unique bounded
solution $\dot \gamma (t)$ up to a multiplicative constant and $DH\big (\gamma
(0)\big )\ne 0$. Then
there is an $\omega _0>0$ such that \eqref{1.2} has, for every $\omega \ge \omega _0,$
an $(\omega,Q)$ affine-periodic solution; these solutions accumulate on
the symmetric heteroclinic cycle $$\bigcup _{i=1}^p\overline {\big
\{Q^i\gamma (t)\mid t\in \R\big \} }$$ as $\omega \to \infty
$.

This result generalizes an accumulation of periodic solutions to\\
homoclinics/heteroclinics in one-degree of freedom conservative
equations like in the above-mentioned Duffing equation $\ddot
x-x+2x^3=0$ to symmetric higher-degree of freedom conservative systems.

\section{Conclusions and final remarks}

In this research paper, we have investigated various classes of multi-dimensional $\rho$-almost periodic type functions and multi-dimensional $(\omega,\rho)$-almost periodic type functions with values in complex Banach spaces. 
We have clarified many structural properties for the introduced classes of functions, which seem to be not considered elsewhere even in the one-dimensional setting. 
Some applications of our results to
the abstract Volterra integro-differential equations and the ordinary differential equations are presented, as well.

Finally, we would like to note that, in this research study, we have been particularly interested in the framework of
continuous and measurable function spaces; the $\rho$-almost
periodicity within the theory of vector-valued (ultra-)distributions and the corresponding classes of Stepanov (Weyl, Besicovitch)  $\rho$-almost
periodic functions will be considered
somewhere else. The analysis of semi-$(\rho_{j},{\mathcal B})_{j\in {\mathbb N}_{n}}$-periodic functions and $({\bf \omega}_{j},\rho_{j}; r_{j}, {\mathbb I}_{j}')_{j\in {\mathbb N}_{n}}$-almost periodic type functions will be carried out somewhere else, as well (cf. \cite{c1} and \cite{nds-2021} for more details).

\end{document}